\documentclass[11pt]{article}

\usepackage[utf8x]{inputenc}
\usepackage[english]{babel}
\usepackage[a4paper,margin=2.5cm]{geometry}
\usepackage{charter,tikz,url}
\usepackage[plainpages=false,pdfpagelabels,colorlinks=true,citecolor=blue,hypertexnames=false]{hyperref}
\usepackage{amsmath,amssymb,amsthm}
\usepackage{stmaryrd,dsfont}
\usepackage{enumerate,array}
\usepackage{listings}
\usepackage{comment}

\definecolor{darkgreen}{rgb}{0,0.5,0}

\def\ph{\vphantom{$A^A_A$}}

\newtheorem{theointro}{Theorem}
\newtheorem{theo}{Theorem}[section]
\newtheorem{prop}[theo]{Proposition}
\newtheorem{lem}[theo]{Lemma}
\newtheorem{cor}[theo]{Corollary}
\newtheorem{conj}[theo]{Conjecture}
\theoremstyle{definition}
\newtheorem{deftn}[theo]{Definition}
\newtheorem{rem}[theo]{Remark}

\newcommand{\N}{\mathbb N}
\newcommand{\Z}{\mathbb Z}

\newcommand{\Q}{\mathbb Q}

\newcommand{\F}{\mathbb F}

\newcommand{\R}{\mathbb R}
\newcommand{\m}{\mathfrak m}
\renewcommand{\P}{\mathbb P}

\renewcommand{\L}{\mathcal L}
\renewcommand{\L}{\Omega}

\newcommand{\B}{\mathbb B}
\newcommand{\calB}{\mathcal B}
\newcommand{\Sp}{\mathbb S}

\newcommand{\E}{\mathbb E}

\newcommand{\val}{\text{\rm val}}
\renewcommand{\sp}{\text{\rm sp}}
\renewcommand{\mod}{\:\text{\rm mod}\:}
\newcommand{\piv}{\text{\rm piv}}
\newcommand{\can}{\text{\rm can}}
\newcommand{\dist}{\text{\rm dist}}

\newcommand{\an}{\text{\rm an}}

\newcommand{\Frac}{\text{\rm Frac}\:}
\newcommand{\card}{\text{\rm Card}\:}

\newcommand{\ord}{\text{\rm ord}}

\renewcommand{\leq}{\leqslant}
\renewcommand{\geq}{\geqslant}

\title{Almost all non-archimedean Kakeya sets have measure zero}
\author{Xavier Caruso}
\date\today

\begin{document}

\maketitle

\begin{abstract}
We study Kakeya sets over local non-archimedean fields with a 
probabilistic point of view: we define a probability measure on the 
set of Kakeya sets as above and prove that, according to this measure,
almost all non-archimedean Kakeya sets are neglectable according to 
the Haar measure.
We also discuss possible relations with the non-archimedean Kakeya
conjecture.
\end{abstract}

\setcounter{tocdepth}{2}
\tableofcontents

\bigskip

\noindent\hfill\hrulefill\hfill\null

\bigskip

At the beginning of the 20th century, Kakeya asks how small can be a 
subset of $\R^2$ obtained by rotating a needle of length $1$ continuously through 
$360$ degrees within it and returning to its original position. A set
satisfying the above requirement is today known as a \emph{Kakeya set}
(or sometimes \emph{Kakeya needle set}) in $\R^2$.
In 1928, Besikovitch \cite{besikovitch} constructed a subset of $\R^2$
with Lebesgue measure zero containing a unit length segment in each 
direction and derived from this the existence of Kakeya sets with
arbitrary small positive Lebesgue measure. Since this, Kakeya sets
have received much attention because they have connections
with important questions in harmonic analysis. In particular, lower
bounds on the size of a Kakeya set have been found: it has been notably 
established that any Kakeya set in $\R^2$ must have Hausdorff dimension 
$2$ (see \cite[Theorem~2]{green} for a short and elegant proof).

Kakeya's problem extends readily to higher dimensions: a \emph{Besikovitch
set} in $\R^d$ is a subset of $\R^d$ containing a unit length segment in
each direction while a \emph{Kakeya set} in $\R^d$ is a set obtained by
rotating \emph{continuously} a needle of length $1$ is all directions
(parametrized either by the $(d{-}1)$-dimensional sphere of the 
$(d{-}1)$-dimensional projective space). We refer to \S \ref{sssec:R}
for precise definitions. Kakeya sets in $\R^d$ with Lebesgue measure
zero exist as well: the product of $\R^{d-2}$ by a neglectable Kakeya 
set in $\R^2$ makes the job. As for lower bounds, it has been proved
by Wolff \cite{wolff} that a Kakeya set in $\R^d$ has Hausdorff dimension 
$\geq \frac{d+2}2$. More recently Katz and Tao \cite{katz} improved the
lower bound to $(2 - \sqrt 2)(d-4) + 3$.
Experts however believe that these results are far from
being optimal and actually conjecture that a Kakeya set in $\R^d$ should 
always have Hausdorff dimension $d$: this is the so-called Kakeya 
conjecture.

\medskip

More recently Kakeya's problem was extended over other fields. The first 
case of interest was that of finite fields and was first considered in 
\cite{wolff} by Wolff. Given a finite field $\F_q$, a \emph{Besikovitch 
set} in $\F_q^d$ is a subset of $\F_q^d$ containing an affine line in 
each direction (note that the length condition has gone). Wolff wondered 
whether there exists a positive constant $c_d$ depending only $d$ such
that any Besikovitch set in $\F_q^d$ contains at least $c_d \cdot q^d$ 
elements. A positive answer (leading to $c_d = \frac 1{d!}$) was given 
by Dvir in his famous paper \cite{dvir}.

In \cite{ellenberg}, Ellenberg, Oberlin and Tao introduced Besikovitch 
sets over $\F_q[[t]]$ and asked whether there exists such a set whose 
Haar measure is zero. Dummit and Hablicsek addressed this question in 
\cite{dummit} and gave to it a positive answer: they proved that, for 
all $d \geq 2$ and all finite field $\F_q$, there does exist a 
zero-measure Besikovitch set in $\F_q[[t]]^d$. They more generally 
defined Besikovitch sets over any ring $R$ admitting a Haar measure 
$\mu$ for which $\mu(R)$ is finite and, for those rings, they stated
a straightforward analogue of the Kakeya conjecture.
Apart from $\F_q[[t]]$, an interesting ring $R$ which falls within 
Dummit and Hablicsek's framework is $R = \Z_p$, the ring of $p$-adic 
integers. Dummit and Hablicsek then proved the Kakeya conjecture in 
dimension $2$ for $R = \F_q[[t]]$ and $R = \Z_p$. The existence of
zero-measure Besikovitch sets over $\Z_p$ was proved more recently
by Fraser in \cite{fraser}.

\medskip

The general aim of this paper is to study further the size of 
Kakeya/Besikovitch sets over non-archimedean local fields, \emph{i.e.} 
$\F_q((t)) = \Frac \F_q[[t]]$, $\Q_p = \Frac \Z_p$ and their extensions. 
Our main originality is that we adopt a probabilistic point of view. 

Let us describe more precisely our results. Let $K$ be a fixed 
non-archimedean local field: similarly to $\R$, it is equipped with an 
absolute value which turns it into a topological locally compact field. 
It is thus equipped with a Haar measure $\mu$ giving a finite mass to 
any bounded subset. Since $K$ is non-archimedean, the unit ball $R$ of 
$K$ is a subring of $K$ (it is $\F_q[[t]]$ when $K = \F_q((t))$ and 
$\Z_p$ when $K = \Q_p$); we normalize $\mu$ so that $\mu(R) = 1$. In
this setting, we provide a definition for Kakeya sets and Besikovitch 
sets\footnote{We emphasize that, similarly to the real setting, we
make the difference between Kakeya and Besikovitch sets: basically,
an additional continuity condition (corresponding to the fact that
Kakeya's needle has to move continuously) is required for the former.} 
and endow the set of Kakeya sets included in $R^d$ with a probability 
measure, giving this way a precise sense to the notion of random 
non-archimedean Kakeya set. Our main theorem is the following.

\begin{theointro}[\emph{cf} Corollary~\ref{cor:measure}]
\label{theointro:measure}
Almost all Kakeya sets sitting in $R^d$ have measure zero.
\end{theointro}

We emphasize that the above theorem concerns actual Kakeya sets (and not 
Besikovitch sets). It then shows a clear dichotomy between the 
archimedean and the non-archimedean setup: in the former, Kakeya sets 
have necessarily positive measure (through it can be arbitrarily small) 
while, in the latter, almost all of them have measure zero.

We will deduce Theorem~\ref{theointro:measure} from a much more accurate 
result providing an \emph{exact} value for the average size of 
$\varepsilon$-neighbourhoods of Kakeya sets. Before stating (a weak 
version of) it, recall that the $\varepsilon$-neighbourhood of a subset 
$N \subset K^d$ consists of points whose distance to $N$ is at most 
$\varepsilon$. Let $q$ be the cardinality of the residue field of $K$, 
\emph{i.e.} of $q = \card R/\m$ where $\m$ is the open unit ball in $K$. 
When $R = \F_q((t))$, we can check that $\m = t \cdot \F_q[[t]]$, so 
that $q$ is indeed $q$. Similarly when $R = \Q_p$, we have $\m = p\Z_p$ 
and $q$ is equal to $p$.

\begin{theointro}[\emph{cf} Proposition~\ref{prop:equiv}] 
\label{theointro:equiv}
The expected value of the Haar measure of the $\varepsilon$-neighbourhood 
of a random Kakeya set sitting in $R^d$ is equivalent to:
$$\frac {2 \cdot (q^d - 1)}{(q-1)(q^{d-1}-1)}
\cdot \frac 1 {|\log_q \varepsilon|}$$
when $\varepsilon$ goes to $0$.
\end{theointro}

This refined version of Theorem~\ref{theointro:measure} seems to us 
quite interesting because it underlines that, although Kakeya sets tend 
to be neglectable according to the Haar measure, they are not that small 
on average as reflected by the logarithmic decay with respect to 
$\varepsilon$. In particular Theorem~\ref{theointro:equiv} is in line 
with the non-archimedean Kakeya conjecture and might even be thought (with 
caution) as an average version of it.

Of course, beyond the mean, one would like to study further
the random variables $X_\varepsilon$ taking a random Kakeya set sitting 
in $R^d$ to the Haar measure of its $\varepsilon$-neighbourhood. For 
instance, in the direction of the non-archimedean Kakeya conjecture, one may 
ask the following question: can one compute higher moments of the 
$X_\varepsilon$'s (possibly extending the technics of this paper) and 
this way derive interesting informations about their minimum? In the 
real setting, results in related directions were obtained by Babichenko 
and al. \cite[Theorem~1.6]{babichenko} in the $2$-dimensional case.

\bigskip

This paper is organized as follows. 
In Section \ref{sec:results}, we define non-archimedean Kakeya/Besikovitch 
sets together with the probability measure on the set of Kakeya sets we 
shall work with afterwards. We then state (without proof) our main 
theorem which is yet another refined version of 
Theorem~\ref{theointro:equiv}. We then derive from it several 
corollaries.
Section \ref{sec:algebraic} provides a totally algebraic reformulation 
of the statements and results of Section \ref{sec:results}. Its interest 
is twofold. First it allows us to extend to the torsion case the notion 
of Kakeya/Besikovitch sets together with the Kakeya conjecture. Second 
it positions the framework in which the forthcoming proof will all take 
place.
The proof of our main theorem occupies Section \ref{sec:proofs}.
Section \ref{sec:numeric} contains numerical simulations whose 
objectives are, first, to exemplify our results and, second, to show the 
behaviour of the random variables $X_\varepsilon$'s beyond their mean. 
Pictures of $2$-adic Kakeya sets (in dimension $2$ and $3$) are also 
included.

\section{Non-archimedean Kakeya sets}
\label{sec:results}

As just mentionned, the aim of this section is to introduce (random) 
Kakeya and Besikovitch sets over non-archimedean local fields (\emph{cf} 
\S\S \ref{ssec:definitions}--\ref{ssec:universe}) and then to state and 
comment on our main results (\S \ref{ssec:statement}).

Throughout this paper, the letter $K$ refers to a fixed discrete 
valuation field on which the valuation is denoted by $\val$. We always 
assume that $K$ is complete and that its residue field is finite. For 
our readers who are not familiar with non-archimedean geometry, we refer 
to Appendix~\ref{app:DVF} (page \pageref{app:DVF}) for basic definitions 
and basic facts about valuation fields.

We fix in addition an integer $d \geq 2$: the dimension.

\subsection{Besikovitch and Kakeya sets}
\label{ssec:definitions}

\subsubsection{The real setting}
\label{sssec:R}

We first recall the definition and the basic properties of Kakeya sets 
and Besikovitch sets in the classical euclidean setting over $\R$. Let 
$\Sp^{d-1}(\R)$ denote the unit sphere in $\R^d$.
When $d=2$, Kakeya considers subsets in $\R^2$ that can be obtained by 
rotating a needle of length $1$ continuously through $360$ degrees 
within it and returning to its original position (or, depending on 
authors, by rotation a needle of length $1$ continuously through $180$ 
degrees within it and to its original position with reverse 
orientation). This notion can be extended to higher dimensions as 
follows.

\begin{deftn}
A \emph{Kakeya needle set} (or just a \emph{Kakeya set}) in $\R^d$ is
a subset $N$ of $\R^d$ of the form:
$$N = \bigcup_{a \in \Sp^{d-1}(\R)} 
\Big[f(a){-}\frac a 2,\, f(a){+}\frac a 2\Big]$$
where $f : \Sp^{d-1}(\R) \to \R^d$ is a continuous function. (Here
$[x,y]$ denotes the segment joining the points $x$ and $y$.)
\end{deftn}

\begin{rem}
\label{rem:projKakeya}
Optionally one may further require that the segments corresponding to 
the directions $a$ and $-a$ coincide for all $a \in \Sp^{d-1}(\R)$. 
This is equivalent to requiring that $f(a) = f(-a)$ for all $a \in 
\Sp^{d-1}(\R)$, that is to requiring that $f$ factors through the
projective space $\P^{d-1}(R)$.
\end{rem}

The natural question about Kakeya sets is the following: how small can 
be a Kakeya set? As a basic example, Kakeya first asks whether there 
exists a minimal area for Kakeya sets in $\R^2$. Besikovitch answers 
this question negatively and proves that there exists Kakeya sets (in
any dimension) of arbitrary small measure.
Besikovitch introduced a weaker version of Kakeya sets:

\begin{deftn}
A \emph{Besikovitch set} in $\R^d$ is a subset of $\R^d$ which contains
a unit line segment in every direction.
\end{deftn}

\noindent
Obviously a Kakeya set is a Besikovitch set. The converse is however
not true. More precisely Besikovitch managed to construct Besikovitch
sets of measure zero whereas one can easily show that a Kakeya set 
have necessarily positive measuree. The question now becomes: how small can be a
Besikovitch set? A famous conjecture in this direction asks whether
any Besikovitch set in $\R^d$ has Hausdorff dimension $d$? It is known
to be true when $d \in \{1,2\}$ but the question remains open for
higher dimensions.

\subsubsection{The non-archimedean setting}

We now move to the non-archimedean setting: recall that we have fixed
a complete discrete valuation field $K$.
We denote by $R$ its rings of integers and by $k$ its residue field.
We set $q = \card k$.
We fix a uniformizer $\pi \in K$ and always assume that the valuation
on $K$ is normalized so that $\val(\pi) = 1$.
Let $\mu$ be the Haar measure on $K$ normalized by $\mu(R) = 1$. In the 
sequel, we shall always work with the norm $|{\cdot}|$ on $K$ defined 
by $|x| = q^{-\val(x)}$ ($x \in K$). We recall that it is compatible 
with the Haar measure $\mu$ on $K$ in the sense that:
$$\mu(aE) = |a| \cdot \mu(E)$$
for all $a \in K$ and all measurable subset $E$ of $K$.

We consider the $K$-vector space $K^d$ and endow it with the infinite norm 
$\Vert{\cdot}\Vert_\infty$:
$$\Vert(x_1, \ldots, x_d)\Vert_\infty = \max_{1 \leq i \leq d} |x_i|.$$
Let $\B^d(K)$ (resp. $\Sp^{d-1}(K)$) denote the unit ball (resp. the unit 
sphere) in $K^d$. Clearly $\B^d(K) = R^d$ and $\Sp^{d-1}(K)$ consists of
tuples $(x_1, \ldots, x_d) \in R^d$ containing at least one coordinate
which is invertible in $R$. The latter condition is equivalent to the
fact that the image of $(x_1, \ldots, x_d)$ in $k^d$ does not vanish.
This notably implies that $\Sp^{d-1}(K)$ has a large measure: precisely 
$\mu(\Sp^{d-1}(K)) = 1 - q^{-d}$. This contrasts with the real case.

\begin{deftn}
\label{def:Besikovitch}
Given $a \in \Sp^{d-1}(K)$, a \emph{unit length segment} of direction
$a$ is a subset of $K^d$ of the form $\big\{ t a + b : t\in R \big\}$
for some $b \in R^d$.

A \emph{Besikovitch set} in $K^d$ is a subset of $K^d$ containing a
unit length segment in every direction.
\end{deftn}

\begin{deftn}
A \emph{Kakeya set} in $K^d$ is a subset $N$ of $K^d$ of the form:
$$N = \bigcup_{a \in \Sp^{d-1}(K)} S_a
\quad \text{with} \quad
S_a = \big\{ t a + f(a) \::\: t\in R \big\}$$
where $f : \Sp^{d-1}(K) \to K^d$ is a continuous function.
\end{deftn}

It has been proved recently (see \cite{fraser}) that Kakeya sets of 
measure zero exists in $K^d$! The main objective of this article is to 
prove that it is in fact the case for almost all Kakeya sets (in a sense 
that we will make precise later).

\subsection{The projective space over $K$}
\label{ssec:projective}

Instead of working with $\Sp^{d-1}(K)$, it will be more convenient to use 
the projective space $\P^{d-1}(K)$. Recall that is defined as the set of 
lines in $K^d$ passing through the origin. From an algebraic point of 
view, $P^{d-1}(K)$ is described as the quotient of $K^{d+1} \backslash 
\{0\}$ by the natural action by multiplication of $K^\star$.
We use the standard notation $[a_1 : \cdots : a_d]$ to refer to the
class in $\P^{d-1}(K)$ of a nonzero $d$-tuple $(a_1, \ldots, a_d)$ of
elements of $K$. Geometrically $[a_1 : \cdots : a_d]$ corresponds to the 
line directed by the vector $(a_1, \ldots, a_d)$.

\begin{deftn}
Let $a \in \P^{d-1}(K)$.
A representative $(a_1, \ldots, a_d) \in K^d$ of $a$ is 
\emph{reduced} if it belongs to $\Sp^{d-1}(K)$.
\end{deftn}

Any element $a \in \P^{d-1}(K)$ admits a reduced representative: it can be 
obtained by dividing any representative $(a_1, \ldots, a_d)$ by a 
coordinate $a_i$ for which $\Vert (a_1, \ldots, a_d) \Vert_\infty = 
|a_i|$. We note that two reduced representatives of $a$ differ by 
multiplication by a scalar of norm $1$, \emph{i.e.} by an invertible 
element of $R$. As a consequence $\P^{d-1}(K)$ can alternatively be 
described as the quotient $\Sp^{d-1}(K)/R^\times$ where $R^\times$ stands 
for the group of invertible elements of $R$.

\paragraph{Canonical representatives.}

Although there is no canonical choice, we will need to define a 
particular set of representatives of the elements of $\P^{d-1}(K)$. The 
following lemma makes precise our convention.

\begin{lem}
Any element $a \in \P^{d-1}(K)$ admits a unique representative
$\can(a) = (a_1, \ldots, a_d) \in \Sp^{d-1}(K)$ satisfying the following 
property: there exists an index $\piv(a)$ (uniquely determined) such 
that $a_{\piv(a)} = 1$ and $|a_i| < 1$ for all $i < \piv(a)$.
\end{lem}

\begin{proof}
Let $(a'_1, \ldots, a'_d) \in \Sp^{d-1}(K)$ be any representative of $a$ 
of norm $1$. Define $j$ as the smallest index $i$ for which $|a'_i|
= 1$. Then the vector
$(a'_j)^{-1} \cdot (a'_1, \ldots, a'_d)$
satisfies the requirements of the lemma (with $\piv(a) = j$). 
The uniqueness is easy and left to the reader.
\end{proof}

\begin{rem}
The notation $\piv$ means ``pivot''.
\end{rem}

The above construction defines two mappings $\piv : \P^{d-1}(K) \to 
\{1,\ldots,d\}$ and $\can : \P^{d-1}(K) \to \Sp^{d-1}(K)$ and the latter 
is a section of the projection $\Sp^{d-1}(K) \to \P^{d-1}(K)$. In the
sequel, we shall often consider $\can$ as a function from $\P^{d-1}(K)$
to $R^d$.

\paragraph{A distance on $\P^{d-1}(K)$.}

Recall that we have seen that $\P^{d-1}(K) = \Sp^{d-1}(K)/R^\times$.
The natural distance on $\Sp^{d-1}(K)$ (inherited from that on $K^d$)
then defines a distance $\dist$ on $\P^{d-1}(K)$ by:
$$\dist(a,b) = \inf_{\hat a, \hat b} |\hat a - \hat b|$$
where the infimum is taken over all representatives $\hat a$ and 
$\hat b$ of $a$ and $b$ respectively \emph{lying in $\Sp^{d-1}(K)$}.
One easily proves that $\dist$ takes its values in the set $\{0, 1, 
q^{-1}, q^{-2}, q^{-3}, \ldots \}$ and remains non-archimedean in the sense 
that
$$\dist(a,c) \leq \max\big(\dist(a,b), \dist(b,c)\big)$$
for all $a, b, c \in \P^{d-1}(K)$. Moreover $\P^{d-1}(K)$ equipped
with the topology induced by $\dist$ is a compact space since there
is a continuous map $\Sp^{d-1}(K) \to \P^{d-1}(K)$ with compact
domain.

\begin{prop}
\label{prop:distproj}
For all $a,b \in \P^{d-1}(K)$, we have:
$$\dist(a,b) = |\can(a) - \can(b)|.$$
\end{prop}

\begin{proof}
Clearly $\dist(a,b) \leq |\can(a) - \can(b)|$.

Hence, we just need to prove that $|\can(a) - \can(b)| \leq 
\dist(a,b)$. Let us first assume that $\piv(a) < \piv(b)$ and let $\hat 
a = (\hat a_1, \ldots, \hat a_d)$ and $\hat b = (\hat b_1, \ldots, \hat 
b_d)$ be two vectors in $\Sp^{d-1}(K)$ lifting $a$ and $b$ 
respectively. Set $j = \piv(a)$. The coordinate $\hat a_j$ has 
necessarily norm $1$ while $|\hat b_j| < 1$. Therefore $|\hat a_j-\hat 
b_j|$ has norm $1$ and $\dist(a,b)$ is equal to $1$ as well.
We conclude similarly when $\piv(a) > \piv(b)$.

Assume now that $\piv(a) = \piv(b)$. Set $j = \piv(a)$ and write
$\can(a) = (a_1, \ldots, a_d)$ and $\can(b) = (b_1, \ldots, b_d)$,
so that $a_j = b_j = 1$.
We notice that any representative $\hat a \in \Sp^{d-1}(K)$ of $a$ 
can be written $\hat a = \lambda \cdot \can(a)$ for some $\lambda
\in R^\times$. Similarly we can write $\hat b = \mu \cdot \can(b)$ 
with $\mu \in R^\times$ for any representative $\hat b$ of $b$. We
are then reduced to show that:
\begin{equation}
\label{eq:minorcan}
|\lambda \cdot \can(a) - \mu \cdot \can(b)|
\geq |\can(a) - \can(b)|
\end{equation}
for any $\lambda$ and $\mu$ of norm $1$. Set $r = |\can(a) - \can(b)|$.
Observe that the $j$-th 
coordinate of the vector $\lambda \cdot \can(a) - \mu \cdot \can(b)$
is $\lambda - \mu$. The inequality~\eqref{eq:minorcan} then holds if
$|\lambda - \mu| \geq r$. Otherwise, let $j'$ be an index such that 
$r = |a_{j'} - b_{j'}|$. For this particular $j'$, write
$\lambda a_{j'} - \mu b_{j'}
= \lambda (a_{j'} - b_{j'}) + (\lambda - \mu) b_{j'}$.
Moreover $|\lambda (a_{j'} - b_{j'})| = r$ while 
$|(\lambda - \mu) b_{j'}| \leq |\lambda - \mu| < r$. 
Thus $|\lambda a_{j'} - \mu b_{j'}| = r$ and \eqref{eq:minorcan}
follows.
\end{proof}

\begin{cor}
\label{cor:distproj}
Let $a,b \in \P^{d-1}(K)$. 
Let $(a_1, \ldots, a_d)$ and $(b_1, \ldots, b_d)$ in $\Sp^{d-1}(K)$ 
be some representatives of $a$ and $b$ respectively. Then $\dist(a,b)$ 
is the maximal norm of a $2 \times 2$ minor of the matrix
\begin{equation}
\label{eq:Mab}
\left( \begin{matrix}
a_1 & a_2 & \cdots & a_d \\
b_1 & b_2 & \cdots & b_d 
\end{matrix} \right).
\end{equation}
\end{cor}

\begin{proof}
Since two representatives of $a$ differ by multiplication by an
element of norm $1$, we may safely assume that $(a_1, \ldots, a_d)
= \can(a)$. Similarly we assume that $(b_1, \ldots, b_d) = \can(b)$.
If $\piv(a) \neq \piv(b)$, the determinant of the submatrix 
of~\eqref{eq:Mab} composed by the $\piv(a)$-th and $\piv(b)$-th columns 
is congruent to $\pm 1$ modulo $\m$. It thus has norm $1$ and the
corollary is proved in this case.
Suppose now that $\piv(a) = \piv(b)$ and assume further for simplicity 
that they are equal to $1$. The matrix~\eqref{eq:Mab} is then equivalent 
to:
$$\left( \begin{matrix}
1 & a_2 & \cdots & a_d \\
0 & b_2-a_2 & \cdots & b_d-a_d 
\end{matrix} \right).$$
It is now clear that the maximal norm of a $2 \times
2$ minor is equal to $\Vert\can(b) - \can(a)\Vert_\infty$. The
corollary then follows from Proposition \ref{prop:distproj}.
\end{proof}

\paragraph{Projective Kakeya sets.}

Following Remark~\ref{rem:projKakeya}, one may define non-archimedean 
Kakeya sets using the projective space instead of the sphere.

\begin{deftn}
A \emph{projective Kakeya set} in $K^d$ is a subset $N$ of $K^d$ of the 
form:
$$N = \bigcup_{a \in \P^{d-1}(K)} S_a
\quad \text{with} \quad
S_a = \big\{ t \cdot \can(a) + f(a) \::\: t\in R \big\}$$
where $f : \P^{d-1}(K) \to K^d$ is a continuous function.
\end{deftn}

\begin{prop}
\begin{enumerate}[(a)]
\setlength\itemsep{0pt}
\item Any projective Kakeya set is a Kakeya set.
\item Any Kakeya set contains a projective Kakeya set.
\end{enumerate}
\end{prop}

\begin{proof}
\emph{(a)}~The projective Kakeya set attached to a function
$f : \P^{d-1}(K) \to K^d$ is equal to the Kakeya set attached to the
compositum of $f$ with the natural map $\Sp^{d-1}(K) \to \P^{d-1}(K)$
sending a vector to the line it generates.

\smallskip

\noindent
\emph{(b)}~The Kakeya set attached to a function $f : \Sp^{d-1}(K) \to K^d$ 
contains the projective Kakeya set attached to $f \circ \can$. (Notice
that $\can$ is continuous by Proposition~\ref{prop:distproj}.)
\end{proof}

In what follows, we will mostly work with projective Kakeya sets.

\subsection{The universe}
\label{ssec:universe}

To each continuous function $f : \P^{d-1}(K) \to K^d$, we attach the
(projective) Kakeya set $N(f)$ defined by:
$$N(f) = \bigcup_{a \in \P^{d-1}(K)} S_a(f)
\quad \text{with} \quad
S_a(f) = \big\{ t \cdot \can(a) + f(a) \::\: t\in R \big\}.$$
Observe that $N(f)$ is compact. Indeed it appears as the image of 
the compact space $\P^{d-1}(K) \times R$ under the continuous mapping 
$(a,t) \mapsto t \cdot \can(a) + f(a)$. In particular, it is closed in
$K^d$.

We would like to define random Kakeya sets, that is to turn $N$ into a 
random variable on a certain probability space $\Omega$. Of course, the 
whole set $C^0(\P^{d-1}(K), K^d)$ of all continuous functions 
$\P^{d-1}(K) \to K^d$ cannot be endowed with a nice probability measure 
because $K$ itself cannot. We then need to restrict the codomain and a 
second natural candidate for $\Omega$ is then $C^0(\P^{d-1}(K), R^d)$. 
Unfortunately, we were not able to find a reasonable definition of a 
probability measure on it\footnote{By the way, it would be interesting to 
define a nice probability measure on $C^0(R, R^d)$ and then investigate 
what could be the non-archimedean analogue of the Brownian motion.}. 
Nevertheless our intuition is that $C^0(\P^{d-1}(K), R^d)$ would be in 
any case too large to be relevant for the application we have in mind; 
indeed, we believe that any reasonable probability measure on it (if it 
exists) would eventually lead to $N(f) = R^d$ almost surely.

Instead, we propose to define $\Omega$ as the set of $1$-Lipschitz 
functions from $\P^{d-1}(K)$ to $R^d$. The addition on $R^d$ turns 
$\Omega$ into a commutative group. We endow $\Omega$ with the infinite 
norm $\Vert \cdot \Vert_\infty$ defined by the usual formula:
$$\Vert f \Vert_\infty = \sup_{a \in \P^{d-1}(K)} \Vert f(a)\Vert_\infty
\quad (f \in \Omega).$$
The induced topology is then the topology of uniform convergence. 
The Arzelà--Ascoli theorem implies that $\Omega$ is compact. It is thus
endowed with its Haar measure, which is a probability measure.

\begin{rem}
More generally, one could also have considered $r$-Lipschitz functions 
$\P^{d-1}(K) \to R^d$ for some positive \emph{fixed} real number $r$.
This would actually lead to similar qualitative behaviours (although of 
course precise numerical values would differ). Moreover the technics 
introduced in this paper extends more or less easily to the general case 
--- and the reader is invited to write it down as an exercise! 
We have chosen to restrict ourselves to the case $r=1$ in order to avoid
many technicalities and be able to focus on the heart of the argumentation.
\end{rem}

In the rest of this paragraph (which can be skipped on first reading), 
we give a more explicit description of the universe $\Omega$ as a 
probability space. We fix a complete set of representatives of classes 
modulo $\m$ and call it $S$. We denote by $S_n$ the set of elements that 
can be written as $$s_0 + s_1 \pi + s_2 \pi^2 + \cdots + s_{n-1} 
\pi^{n-1}$$ where the $s_i$'s lie in $S$ and we recall that $\pi \in R$ 
denotes a fixed uniformizer of $K$. Then $S_n$ forms a complete set of 
representatives of classes modulo $\m^n$. Observe in particular that 
$S_1 = S$.

\noindent
We now introduce special ``step functions'' that will be useful for 
approximating functions in $\Omega$.

\begin{deftn}
For a positive integer $n$, let $\Omega^\an_n$ denote the subset of 
$\Omega$ consisting of functions taking their values in $S_n$ and which 
are constant of each closed ball of radius $q^{-n}$.
\end{deftn}

\begin{rem}
The exponent ``$\an$'' refers to ``analytic'' and recalls that we
are here giving an analytic description of $\Omega$. Later on, in \S
\ref{ssec:universealg}, we will revisit the constructions of this 
subsection in a more algebraic fashion and notably define an algebraic
version of $\Omega^\an_n$.
\end{rem}

Note that $\Omega^\an_n \subset \Omega^\an_m$ as soon as $n \leq m$. 
Moreover $\Omega^\an_n$ is a finite set. Indeed $S_n$ is finite and the 
set of closed balls of radius $q^{-n}$ is in bijection with 
$\P^{d-1}(S_n)$ and thus is finite as well.

\begin{prop}
\label{prop:approx}
Given $n \geq 1$ and $f \in \Omega$
there exists a unique function $\psi^\an_n(f) \in \Omega^\an_n$ such that:
$$\Vert f - \psi^\an_n (f) \Vert_\infty \leq q^{-n}.$$
\end{prop}

\begin{proof}
Let $a \in \P^{d-1}(K)$. Set $f(a) = (x_1, \ldots , x_d)$ where the 
$x_i$'s lie in $R$. For any $i$, let $y_i$ be the
unique element of $S_n$ which is congruent to $x_i$ modulo $\m^n$. 
We define $\psi^\an_n(f)(a) = (y_1, \ldots, y_d)$.
Remembering that $\Vert x - y \Vert_\infty \leq q^{-n}$ (with $x, y \in R^d$)
if and only if $x$ and $y$ are congruent modulo
$\m^n$ coordinate-wise, we deduce that $\psi^\an_n(f)(a)$ is the unique element 
of $S_n$ with the property
that:
$$\Vert f (a) - \psi^\an_n (f )(a) \Vert_\infty \leq q^{-n}.$$
This construction then defines a function $\psi^\an_n(f ) : \P^{d-1} (K) \to S_n$
such that $\Vert f - \psi^\an_n(f ) \Vert_\infty \leq q^{-n}$.
and we have shown in addition that $\psi^\an_n(f)$ is the unique function 
satisfying the above condition.

It then remains to prove that $\psi^\an_n (f) \in \Omega^\an_n$, \emph{i.e.} 
that (1)~$\psi^\an_n(f)$ is constant on each closed ball of radius $q^{-n}$ 
and (2) is $1$-Lipschitz. Let us first prove (1). Let $a, b \in \P^{d-1} 
(K)$ such that $\dist(a, b) \leq q^{-n}$. By the Lipschitz condition, we 
get $\Vert f (a) - f (b) \Vert_\infty \leq q^{-n}$ as well. In other 
words, $f(a)$ and $f(b)$ are congruent modulo $\m^n$ coordinate-wise. By 
construction of $\psi^\an_n(f)$, we then derive that $\psi^\an_n(f)(a) = 
\psi^\an_n(f)(b)$ and (1)~is proved.

We now move to~(2). Pick $a, b \in \P^{d-1} (K)$. If $\dist(a, b) \leq
q^{-n}$, then we have just seen that
$\psi^\an_n(a) = \psi^\an_n(b)$. Consequently we clearly have 
$\Vert \psi^\an_n(a) - \psi^\an_n(b) \Vert_\infty \leq \dist(a, b)$. Otherwise, we
can write:
$$\Vert \psi^\an_n (a) - \psi^\an_n (b) \Vert_\infty
\leq \max \big(
\Vert \psi^\an_n (a) - f (a) \Vert_\infty,
\Vert f (a) - f (b) \Vert_\infty,
\Vert \psi^\an_n (b) - f (b) \Vert_\infty \big).$$
Now remark that $\Vert \psi^\an_n (a) - f (a) \Vert_\infty$ and 
$\Vert \psi^\an_n (b) - f (b)\Vert_\infty$ are both not greater than $q^{-n}$ by
construction. They are then \emph{a fortiori} both less than $\dist(a, b)$
by assumption. Moreover since $f$ is $1$-Lipschitz, we have $\Vert 
f (a) - f (b) \Vert_\infty \leq \dist(a, b)$. Putting all together we 
finally derive $\Vert \psi^\an_n(a) - \psi^\an_n(b) \Vert_\infty \leq
\dist(a, b)$ as wanted.
\end{proof}

Proposition \ref{prop:approx} just above shows that the union of all 
$\Omega^\an_n$ are dense in $\Omega$. Moreover, there is a projection 
$\psi^\an_n : \Omega \to \Omega^\an_n$ for any $n\geq 1$. For $m\geq n$, let 
$\psi^\an_{m,n} : \Omega^\an_m \to \Omega^\an_n$ denote the restriction of 
$\psi^\an_n$ to $\Omega^\an_m$.

\begin{prop}
\label{prop:fibrepsi}
Let $n$ be a positive integer and $f_n \in \Omega^\an_n$. 
The fibre of $\psi^\an_{n+1,n}$ over $f_n$ consists
exactly of functions of the shape:
$$f_n + \pi^n g$$
where $g : \P^{d-1} (K) \to S_1^d$ is any function which is constant on 
each closed ball of radius $q^{-(n+1)}$.
\end{prop}

\begin{proof}
We notice first that any function $f_{n+1}$ of the form $f_n + \pi^n g$ 
clearly lies in $\Omega^\an_{n+1}$ and maps to $f_n$ under 
$\psi^\an_{n+1,n}$ because
$$\Vert f_{n+1} - f_n \Vert_\infty = \Vert \pi^n g \Vert_\infty = q^{-n} 
\cdot \Vert g \Vert_\infty \leq q^{-n}.$$
Pick now $f_{n+1} \in \Omega^\an_{n+1}$ such that $\psi^\an_{n+1,n}(f_{n+1}) 
= f_n$. Then $\Vert f_{n+1} - f_n \Vert_\infty \leq q^{-n}$, meaning that
$f_{n+1}$ is congruent to $f_n$ modulo $\m^n$, \emph{i.e.} that there 
exists a function $g : \P^{d-1} (K) \to R^d$ such that
$f_{n+1} = f_n + \pi^n g$. Looking at the shape of the elements of $S_n$ 
and $S_{n+1}$, we deduce that g must take its values in $S_1^d$.
\end{proof}

Let $G^\an$ be the set of functions $\P^{d-1} (K) \to S_1^d$ which are 
constant on each closed ball of radius $q^{-i}$. Applying repeatedly 
Proposition \ref{prop:fibrepsi}, we find that the functions in $\Omega^\an_n$ are 
exactly those that can be written as $\sum_{i=1}^{n} g_i \pi^{i-1}$ 
with $g_i \in G^\an_i$. Moreover this writing is unique. Passing to the 
limit, we find that the functions in $\Omega$ can all be written 
uniquely as an inifinite converging sum $\sum_{i=1}^\infty g_i \pi^{i-1}$ 
with $g_i \in G^\an_i$ as above. In other words there is a bijection:
\begin{equation}
\label{eq:bijOmega}
\begin{array}{rcl}
\displaystyle \prod_{i=1}^\infty G^\an_i & 
  \stackrel{\sim}{\longrightarrow} & \Omega \\
(g_1, g_2, \ldots) & \mapsto & 
  \displaystyle \sum_{i=1}^\infty g_i \pi^{i-1}
\end{array}
\end{equation}
Furthermore, if we endow $G^\an_i$ with the discrete topology, the 
above bijection is an homeomorphism. Since the $G^\an_i$'s are all
finite, we recover that $\Omega$ is compact. Finally, the Haar measure
on $\Omega$ can be described as follows: 
it corresponds under the bijection~\eqref{eq:bijOmega} to 
the product measure on $\prod_{i=1}^\infty G^\an_i$ where each factor 
is endowed with the uniform distribution (it may be seen directly but 
it is also a consequence of Proposition \ref{prop:uniformLn} below).
In other words picking a random element in $\Omega$ amounts to picking
each ``coordinate'' $g_i$ in $G^\an_i$ uniformly and independantly.

\subsection{Average size of a random Kakeya set}
\label{ssec:statement}

For $f \in \Omega$, recall that we have defined a Kakeya set $N(f)$.
Recall that $N(f)$ is closed and remark in addition that $N(f) \subset
R^d$ since $f$ takes its values in $R^d$.
Given an auxiliary positive integer $n$, we introduce the 
$(q^{-n})$-neighbourhood $N_n(f)$ of $N(f)$, that is:
$$N_n(f) = \Big\{ \, x \in R^d \,\, \Big| \,\, 
\inf_{y \in N(f)} |x-y| \leq q^{-n} \, \Big\}$$
and let $X_n(f)$ denote its measure.
This defines a collection of random variables $X_n : \Omega \to \N$ that 
measures the size of $N(f)$. Our main theorem provides an explicit 
formula for their mean. Before stating it, let us recall that $q$ 
denotes the cardinality of the residue field $k$.

\begin{theo}
\label{theo:main}
Let $(u_n)_{n \geq 0}$ be the sequence defined by the recurrence:
$$\begin{array}{rcl}
u_0 = 1 & ; & \displaystyle
u_{n+1} = 1 - \left( 1 - \frac{u_n}{q^{d-1}} \right)^{q^{d-1}}.
\end{array}$$
Then:
$$\E[X_n] = 1 - (1 - u_n)^{1 + q^{-1} + \cdots + q^{-(d-1)}}.$$
\end{theo}

This theorem will be proven in Section \ref{sec:proofs}. For now, we 
would like to comment on it a bit and derive some corollaries. The
first one justifies the title of this article.

\begin{cor}
\label{cor:measure}
The set $N(f)$ has measure zero almost surely.
\end{cor}

\begin{proof}
The sequence $(X_n)_{n \geq 1}$
defines a nonincreasing sequence of bounded random variables 
and therefore converges when $n$ goes to infinity. Set $X = \lim_{n 
\to \infty} X_n$. Noting that the $X_n$'s are all bounded by $1$, it 
follows from the dominated convergence theorem that
$\E[X] = \lim_{n \to \infty} \E[X_n]$.
Observing that
$$\forall x > 0, \quad
\left( 1 - \frac x {q^{d-1}} \right)^{q^{d-1}} > \, 1 - x$$
we deduce that the sequence $(u_n)_{n \geq 1}$ of Theorem 
\ref{theo:main} is decreasing and therefore converges. Furthermore, its 
limit is necessarily $0$. This implies that $\E[X] = 0$. Since $X \geq 0$, 
we deduce that $X = 0$ almost surely. Moreover, for a fixed $f \in 
\Omega$, $X(f)$ is the volume of the $\bigcap_n N_n(f)$ which is equal
to $N(f)$ because the latter is closed.
Therefore $N(f)$ has measure zero almost surely.
\end{proof}

\paragraph{Around Kakeya conjecture.}

In the real setting, the classical Kakeya conjecture asks whether any
Besikovitch set in $\R^d$ has maximal Hausdorff dimension. In the
non-archimedean setting, the analogue of the Kakeya conjecture can be
formulated as follows.

\begin{conj}[Kakeya Conjecture]
\label{conj:Kakeya}
Let $B$ be a bounded Besikovitch set in $K^d$. For any positive integer 
$n$, let $B_n$ be the $(q^{-n})$-neighbourhood of $B$:
$$B_n = \Big\{ \, x \in K^d \,\, \Big| \,\, 
\inf_{y \in B} |x-y| \leq q^{-n} \, \Big\}$$
and $\mu_n$ be its Haar measure. Then $|\log\:\mu_n| = o(n)$ when $n$
goes to infinity.
\end{conj}

\begin{rem}
Using the fact the the balls of radius $q^{-n}$ are pairwise disjoint
in the non-archimedean setting, one derives that the minimal number of balls
needed to cover $B_n$ is $q^{nd}\mu_n$. The Hausdorff dimension of $B$
is then defined by the limit of the sequence:
$$\frac{\log(q^{nd}\mu_n)}{n \: \log q} = d + \frac{\log \: \mu_n}
{n \: \log q}$$
and thus is equal to $d$ if and only if $|\log\:\mu_n| = o(n)$.
\end{rem}

The non-archimedean Kakeya conjecture is known in dimension $2$ thanks 
to the works of Dummit and Hablicsek \cite[Theorem~1.2]{dummit}. It is 
contrariwise widely open in higher dimensions (to our knowledege).
Before going further, we state a slight improvement of Dummit and 
Hablicsek's result.

\begin{theo}
\label{theo:conj2}
Let $B$ be a bounded Besikovitch set in $K^2$. For any positive integer $n$, 
let $B_n$ be the $(q^{-n})$-neighbourhood of $B$ and let $\mu_n$ be
its Haar measure. Then:
$$\mu_n \geq \frac 1 {\frac{q-1}{q+1} \: n + 1}.$$
\end{theo}

We postpone the proof of this theorem to \S \ref{ssec:conjKakeya} 
because it will be convenient to write it down using the algebraic 
framework on which we will elaborate later on. 
Instead let us go back to random non-archimedean Kakeya sets. Studying 
further the asymptotic behaviour of the sequence $(u_n)$ defined in 
Theorem \ref{theo:main}, one can determine an equivalent of the mean of 
the random variable $X_n$.

\begin{prop}
\label{prop:equiv}
We have the equivalent:
$$\E[X_n] \sim \frac {2 \cdot (q^d - 1)}{(q-1)(q^{d-1}-1)} 
\cdot \frac 1 n$$
when $n$ goes to infinity.
\end{prop}

\begin{proof}
A simple computation shows that $u_{n+1} = u_n - c \cdot u_n^2 + o(u_n^2)$
with $c = \frac {q^{d-1} - 1}{2 q^{d-1}}$. For $n \geq 0$, define $w_n = 
\frac 1 {u_n}$, so that we have:
$$w_{n+1} - w_n = \frac{u_n - u_{n+1}}{u_n u_{n+1}} 
 = \frac{c u_n^2 + o(u_n^2)}{u_n u_{n+1}}
 = \frac{c u_n^2 + o(u_n^2)}{u_n^2 - c u_n^3 + o(u_n^3)} = c + o(1).$$
Thus $w_n \sim cn$ and $u_n \sim \frac 1{cn}$. The claimed result then
follows from Theorem \ref{theo:main}.
\end{proof}

It follows from Proposition \ref{prop:equiv} that:
$$-\log \E[X_n] = 
\log n + 
\log\left(\frac{(q-1)(q^{d-1}-1)}{2 \cdot (q^d - 1)}\right)
+ o(1).$$
In particular $|\log \E[X_n]| = o(n)$. Proposition \ref{prop:equiv}
then might be thought as an \emph{average} strong version of Conjecture
\ref{conj:Kakeya}. We remark in addition that the lower bound given by 
Theorem \ref{theo:conj2} is rather close to the expected value of $X_n$ 
provided by Proposition \ref{prop:equiv}: roughly the differ by a factor 
$2$. We then expect the random variables $X_n$ to be quite concentrated 
around their mean. We refer to Section \ref{sec:numeric} for numerical 
simulations supporting further this expectation.

\section{Algebraic reformulation}
\label{sec:algebraic}

The aim of this section is merely to rephrase the constructions, 
theorems and conjectures of Section \ref{sec:results} in the more 
abstract framework of algebra in which the proofs of Section 
\ref{sec:proofs} will be written.

%\subsection{Neighbourhoods and classes modulo $\m^n$}

For any positive integer $n$, set $R_n = R/\m^n = R/\pi^n R$. 
It is a finite 
ring of cardinality $q^n$. Concretely if $S \subset R$ is a set of 
representatives of the quotient $R/\m = k$, any class in $R_n$ is
uniquely represented by an element of the shape:
\begin{equation}
\label{eq:representatives}
s_0 + s_1 \pi + s_2 \pi^2 + \cdots + s_{n-1} \pi^{n-1}
\end{equation}
where the $s_i$'s lie in $S$ and we recall that $\pi$ is a fixed
uniformizer of $R$ (that is a generator of $\m$).
Let $p_n : R^d \to R_n^d$ denote the canonical projection taking a tuple 
$(x_1, \ldots, x_d)$ to its class modulo $\m^n$ (obtained by taking the 
class modulo $\m^n$ of each coordinate separatedly).

\begin{prop}
\label{prop:neighbourhood}
Let $E$ be a subset of $R^d$ and $E_n$ denote its 
$(q^{-n})$-neighbourhood, that is:
$$E_n = \Big\{ \, x \in K^d \,\, \Big| \,\, 
\inf_{y \in E} |x-y| \leq q^{-n} \, \Big\}.$$
Then $E_n = p_n^{-1}(p_n(E))$ and the volume of $E_n$ is:
$$\mu(E_n) = q^{-nd} \cdot \card p_n(E).$$
\end{prop}

\begin{proof}
Notice that, given $x = (x_1, \ldots, x_d)$ and $y = (y_1, \ldots, y_d)$ 
in $R^d$, $\Vert x-y \Vert_\infty \leq q^{-n}$ if and only if $x_i 
\equiv y_i \pmod {\m^n}$ for all $i$. As a consequence, the closed ball 
of radius $q^{-n}$ and centre $x$ is exactly $B_x = p_n^{-1}(p_n(x))$. 
This gives the first assertion of the proposition.

To establish the second assertion, it is enough to prove that each
$B_x$ has volume $q^{-nd}$. Observe that $B_x = B_y + (y-x)$. By the
properties of the Haar measure, we then must have $\mu(B_x) = \mu(B_y)$.
Finally we note that the $B_x$'s are pairwise distinct and cover the
whole space $R^d$ when $x$ runs over the tuples $(x_1, \ldots, x_d)$
where each $x_i$ has the shape~\eqref{eq:representatives}. Since there 
are $q^{nd}$ such elements, we are done.
\end{proof}

\subsection{The torsion Kakeya Conjecture}
\label{ssec:torsKakeya}

Recall that the sphere $\Sp^{d-1}(K)$ --- or equivalenty $\Sp^{d-1}(R)$ --- consists 
of tuples $(x_1, \ldots, x_d) \in R^d$ having one invertible coordinate. 
This algebraic description makes sense for more general rings and allows 
us to define $\Sp^{d-1}(R_n)$ as the set of tuples $(x_1, \ldots, x_d) 
\in R_n^d$ for which $x_i$ is invertible in $R_n$ for some $i$. Note
that an element $x \in R_n$ is invertible if and only if its image in
$R/\m = k$ does not vanish, \emph{i.e.} if and only if $x \not\equiv
0 \pmod \m$. 

We can now extend the definition of a Besikovitch set (\emph{cf} 
Definition \ref{def:Besikovitch}) and the Kakeya conjecture (\emph{cf}
Conjecture \ref{conj:Kakeya}) over $R_n$.

\begin{deftn}
Let $n$ be a positive integer and let $\ell \in \llbracket 0,n\rrbracket$.
Given $a \in \Sp^{d-1}(R_n)$, a \emph{segment of length $q^{-\ell}$} of 
direction $a$ is a subset of $R_n^d$ of the form $\big\{ t a + 
b : t\in \m^\ell \big\}$ for some $b \in R_n^d$.

\noindent
A \emph{$\ell$-Besikovitch set} in $R_n^d$ is a subset of $R_n^d$ 
containing a segment of length $q^{-\ell}$ in every direction.
\end{deftn}

\begin{conj}[Torsion Kakeya Conjecture]
There exists a sequence\footnote{This sequence may a priori depend on 
$K$, $d$ and $\ell$.} of positive real numbers $(\varepsilon_n)_{n \geq 1}$ 
converging to $0$ satisfying the following property: for any $n \geq 1$,
any $\ell \in \llbracket 0,n \rrbracket$ and any $\ell$-Besikovitch set 
$B$ in $R_n^d$, we have:
$$\log_q\:\card B \geq n \cdot (d - \varepsilon_n)$$
where $\log_q$ stands for the logarithm in $q$-basis.
\end{conj}

Theorem \ref{theo:conj2} admits an analogue in the torsion case as
well; it can be formulated as follows.

\begin{theo}
\label{theo:conj2alg}
For any positive integer $n$, any integer $\ell \in \llbracket 0,n
\rrbracket$ and any $\ell$-Besikovitch set $B$ in $R_n^2$, we have:
$$\card B \geq q^{2(n-\ell)} \cdot \frac 1 {\frac{q-1}{q+1} \: n + 1}.$$
\end{theo}

Again, we postpone the proof of this theorem to \S 
\ref{ssec:conjKakeya}. Let us however notice here that it implies 
Theorem \ref{theo:conj2}. Indeed, let $B$ be a bounded Besikovitch set 
in $R^2$. Let $\ell$ be an integer for which $B$ is included in the ball 
of centre $0$ and radius $q^\ell$. Then $\pi^{\ell} B \subset R$ and 
$p_n(\pi^{-\ell} B)$ is a $\ell$-Besikovitch set in $R_n^2$. Therefore, 
according to the above theorem, one must have:
$$\card p_n(\pi^{\ell}B) > q^{2(n-\ell)} \cdot \frac 1 {\frac{q-1}{q+1} \: n + 1}.$$
Combining this with Proposition \ref{prop:neighbourhood}, we get the
result.

\subsection{Algebraic description of the projective space}
\label{ssec:projectivealg}

The projective space $\P^{d-1}(K)$ --- considered as a metric space --- 
which has been introduced in \S \ref{ssec:projective} admits an 
algebraic description as well. In order to explain it, let us first
recall that $\P^{d-1}(K) = \Sp^{d-1}(K)/R^\times$. This allows us
to define a specialization map:
$$\begin{array}{rcl}
\sp_1 : \qquad \P^{d-1}(K) & \longrightarrow & \P^{d-1}(k) \smallskip \\{}
[a_1 : \cdots : a_d] & \mapsto & [\bar a_1 : \cdots : \bar a_d]
\end{array}$$
where $(a_1, \ldots, a_d) \in \Sp^{d-1}(K)$ and $\bar a_i$ denotes the 
image of $a_i$ in $k$. With these notations, the index $\piv(a)$ 
(defined in \S \ref{ssec:projective}) appears as the first index of a 
non-vanishing coordinate of $\sp_1(a)$. We notice that the mapping $\sp_1$ 
is surjective and that the preimage of any point in $\P^{d-1}(k)$ is in 
bijection with $R^{d-1}$. Indeed let us define $\piv_1(\bar a)$ as the 
smallest index of a non-vanishing coordinate of $\bar a$ and consider 
the unique representative $(\bar a_1, \ldots, \bar a_d)$ of $a$ such 
that $\bar a_{\piv_1(\bar a)} = 1$. Choose moreover a lifting $a \in R^d$ 
of $(\bar a_1, \ldots, \bar a_d)$ whose $\piv_1(\bar a)$-th coordinate is 
$1$. We can then define a bijection:
$$\begin{array}{rcl}
H_{\piv_1(\bar a)} & \longrightarrow & \sp_1^{-1}(\bar a) \smallskip \\
x & \mapsto & [ a + \pi x ]
\end{array}$$
where $H_{\piv_1(\bar a)}$ denote the coordinate hyperplane of $R^d$ 
defined by the equation $x_{\piv_n(\bar a)} = 0$. We remark moreover 
that the vectors $a + \pi x$ appearing above are all canonical 
representatives.

More generally, for any positive integer $n$, we define:
$$\P^{d-1}(R_n) = \Sp^{d-1}(R_n) / R_n^\times$$
Given $a \in \P^{d-1}(R_n)$, let $\piv_n(a)$ be the index of the first
invertible coordinate of $a$ and $\can_n(a) \in \Sp^{d-1}(R_n)$ 
be the unique representative of $a$ whose $\piv_n(a)$-th coordinate
is $1$. We have a specialization map of level $n$:
$$\begin{array}{rcl}
\sp_n : \qquad \P^{d-1}(K) & \longrightarrow & \P^{d-1}(R_n) \smallskip \\{}
[a_1 : \cdots : a_d] & \mapsto & [a_1 \mod \m^n: \cdots : a_d \mod \m^n].
\end{array}$$
Again $\sp_n$ is surjective and the preimage of any point $a \in 
\P^{d-1} (R_n)$ is isomorphic to $R^{d-1}$ \emph{via}
$$\begin{array}{rcl}
H_{\piv_n(a)} & \longrightarrow & \sp_n^{-1}(a) \smallskip \\
x & \mapsto & [ \can_n(a) + \pi^n x ]
\end{array}$$
Similarly, given a second integer $m \geq n$,
the reduction modulo $\m^n$ defines a map $\sp_{m,n} : \P^{d-1}(R_m) 
\to \P^{d-1} (R_n)$.
This map is surjective and its fibres are all in bijection 
with $R_{m-n}^{d-1}$. It notably follows from this that:
\begin{equation}
\label{eq:cardP1}
\card \P^{d-1}(R_n) = q^{(d-1)(n-1)} \cdot \frac{q^d-1}{q-1}.
\end{equation}

\begin{prop}
\label{prop:descprojective}
The collection of applications $\sp_n$ induces a bijection:
$$\textstyle
\sp : \P^{d-1}(K) \longrightarrow \varprojlim_n \P^{d-1}(R_n)$$ 
where the codomain is by definition the set of all sequences $(x_n)_{n 
\geq 1}$ with $x_n \in \P^{d-1}(R_n)$ and $\sp_{n+1,n}(x_{n+1}) = x_n$ 
for all $n$.
\end{prop}

\begin{proof}
We define a function $\varphi$ in the opposite direction as follows. Let 
$(x_n)_{n \geq 1}$ be a sequence in $\varprojlim_n \P^{d-1}(R_n)$. 
The compatibility condition implies that $\piv_n(x_n)$ is constant and 
that $\can_n(x_n)$ is the reduction modulo $\m^n$ of $\can_{n+1}(x_{n+1})$.
Therefore the sequence $(\can_n(x_n))_{n \geq 1}$ defines an element
$x \in R^d$. The $\piv_1(x_1)$-th coordinate of $x$ is $1$, so that
$x \in \Sp^{d-1}(K)$. Let define $\varphi((x_n)_{n \geq 1})$ as the
class of $x$ in the projective space $\P^{d-1}(K)$.
It is clear that $\varphi \circ \sp$ and $\sp \circ \varphi$
are both the identity, implying that $\sp$ is a bijection as claimed.
\end{proof}

\paragraph{Algebraic version of the distance.}

For $a, b \in \P^{d-1}(K)$, define $v(a,b)$ as the supremum in $\N \cup 
\{+\infty\}$ of the set consisting of $0$ and the positive integers $n$ 
for which $\sp_n(a) = \sp_n(b)$. 
Thanks to Proposition~\ref{prop:descprojective}, $v(a,b) = +\infty$ if 
and only if $a = b$.

\begin{prop}
\label{prop:dist}
Given $a, b \in \P^{d-1}(K)$, we have $\dist(a,b) = q^{-v(a,b)}$.
\end{prop}

\begin{proof}
Note that $\sp_n(a) = \sp_n(b)$ if and only if $a$ and $b$ have the
same image in $\P^{d-1}(S_n)$, \emph{i.e.} if and only if
$\can(a) \equiv \can(b) \pmod {\m^n}$.
The proposition now follows from Proposition \ref{prop:distproj}.
\end{proof}

More generally, given $a, b \in \P^{d-1}(S_n)$, we define $v_n(a,
b)$ as 
the biggest integer $v \in \{0, 1, \ldots, n\}$ for which $\sp_{n,v}(a) 
= \sp_{n,v}(b)$ (with the convention that $v=0$ always satisfies the 
above requirement). As above $v_n(a,b) = n$ if and only if $a = b$.
A torsion analogue of Proposition~\ref{prop:distproj} then holds.

\begin{prop}
\label{prop:distn}
For $a, b \in \P^{d-1}(R_n)$, we can write:
$$\can_n(b) - \can_n(a) = \pi^{v_n(a,b)} \cdot u$$
where $u$ lies in $R_n^d$ and has at least one invertible coordinate.
\end{prop}

\begin{proof}
It is a simple adaptation of the proof of Proposition~\ref{prop:distproj}.
\end{proof}

\subsection{Algebraic description of the universe}
\label{ssec:universealg}

Recall that we have defined in \S \ref{ssec:universe} the set $\Omega$ 
(our universe) consisting of $1$-Lipschitz functions $\P^{d-1}(K) \to 
R^d$. The aim of this subsection is to revisit constructions and results 
of \S \ref{ssec:universe} with an algebraic point of view. We recall 
that we have defined specialization maps $\sp_n : \P^{d-1}(K) \to 
\P^{d-1}(S_n)$ in \S \ref{ssec:projective} and, similarly, that we 
have introduced previously the projections $p_n : R^d \to R_n^d$ taking 
a tuple to its reduction modulo $\m^n$.

The algebraic analogue of the existence of $\psi^\an_n(f)$ can be
formulated as follows.

\begin{prop}
\label{prop:psin}
Let $f \in \Omega$. For all positive integer $n$, there exists a unique 
function $\psi_n(f) : \P^{d-1}(R_n) \to R_n^d$ making the following 
diagram commutative:

\begin{equation}
\label{eq:diagfn}
\raisebox{-0.5\height}{%
\begin{tikzpicture}[xscale=3,yscale=1.5]
\node (P) at (0,0) { \ph $\P^{d-1}(K)$ };
\node (R) at (1,0) { \ph $R^d$ };
\node (Pn) at (0,-1) { \ph $\P^{d-1}(R_n)$ };
\node (Rn) at (1,-1) { \ph $R_n^d$ };
\draw[->] (P)--(R) node[above,midway,scale=0.8] { $f$ };
\draw[->] (P)--(Pn) node[left,midway,scale=0.8] { $\sp_n$ };
\draw[->] (R)--(Rn) node[right,midway,scale=0.8] { $p_n$ };
\draw[->] (Pn)--(Rn) node[above,midway,scale=0.8] { $\psi_n(f)$ };
\end{tikzpicture}}
\end{equation}
\end{prop}

\begin{rem}
Roughly speaking, the function $\psi_n(f)$ encodes the action of 
$\psi^\an_n(f)$ on closed balls of radius $q^{-n}$.
\end{rem}

\begin{proof}[Proof of Proposition \ref{prop:psin}]
The proposition can be derived from Proposition \ref{prop:approx}.
We nevertheless prefer giving an independant and completely algebraic
proof.

Let $a, b \in \P^{d-1}(K)$ with $\sp_n(a) = \sp_n(b)$. By definition 
$\dist(a,b) \leq q^{-n}$. Thus $|f(a) - f(b)| \leq q^{-n}$ 
because $f$ is assumed to be $1$-Lipschitz. Thus $f(a)$ and $f(b)$ lie 
in the same ball of radius $q^{-n}$ or, equivalently, $p_n \circ f(a) = 
p_n \circ f(b)$. In other words, for $x \in \P^{d-1}(K)$, $p_n \circ 
f(x)$ depends only on $\sp_n(x)$. This implies the existence of the 
required mapping $f_n$. The unicity follows from the surjectivity of 
$\sp_n$.
\end{proof}

We emphasize that, we have not proved yet that $\psi_n(f)$ is 
$1$-Lipschitz. Indeed this notion has not been defined yet. Here
is the definition we will use.

\begin{deftn}
A function $f : \P^{d-1}(S_n) \to R_n^d$ is $1$-Lipschitz 
if for all $a,b \in \P^{d-1}(S_n)$:
$$f(a) \equiv f(b) \pmod{\m^{v_n(a,b)}}$$
where the above condition means that all the coordinates of $f(b_n)-
f(a_n)$ lie in $\m^{n-v_n(a,b)}$.

We denote by $\Omega_n$ their set.
\end{deftn}

We notice that $\L_1$ is the of \emph{all} set theoretical functions 
$\P^{d-1}(k) \to k^d$. Moreover, two integers $m \geq n$ together with a 
function $f_m \in \Omega_m$, it is easily checked that there exists a 
unique function $\psi_{m,n}(f_{n+1}) \in \L_n$ making the diagram below 
commutative:

\medskip

\noindent\hfill
\begin{tikzpicture}[xscale=4,yscale=1.5]
\node (P) at (0,0) { \ph $\P^{d-1}(R_m)$ };
\node (R) at (1,0) { \ph $R_m^d$ };
\node (Pn) at (0,-1) { \ph $\P^{d-1}(R_n)$ };
\node (Rn) at (1,-1) { \ph $R_n^d$ };
\draw[->] (P)--(R) node[above,midway,scale=0.8] { $f_m$ };
\draw[->] (P)--(Pn) node[left,midway,scale=0.8] { $\sp_{m,n}$ };
\draw[->] (R)--(Rn) node[right,midway,scale=0.8] { $p_{m,n}$ };
\draw[->] (Pn)--(Rn) node[above,midway,scale=0.8] { $\psi_{m,n}(f_m)$ };
\end{tikzpicture}
\hfill\null

\begin{lem}
\label{lem:comppsin}
The function $\psi_n$ takes its values in $\L_n$.
\end{lem}

\begin{proof}
Let $f \in \Omega$.
Let $a_n, b_n \in \P^{d-1}(S_n)$. If $a_n = b_n$, we have $v_n(a_n,b_n) 
= n$ and there is nothing to prove. Otherwise, pick $a,b \in \P^{d-1}(K)$
such that $\sp_n(a) = a_n$ and $\sp_n(b) = b_n$. Then $\psi_n(f)$ maps
$a_n$ and $b_n$ to $p_n(f(a))$ and $p_n(f(b))$ respectively. We then
need to prove that $p_n(f(a)) - p_n(f(b)) = p_n(f(a) - f(b))$ has all
its coordinates in $\m^{v_n(a_n,b_n)}$. But, using that $f$ is
$1$-Lipschitz, we get:
$$\Vert f(a) - f(b) \Vert_\infty \leq \dist(a,b) = q^{-v(a,b)} = 
q^{-v_n(a_n,b_n)}$$
and we are done.
\end{proof}

One important benefit of working with $\L_n$ instead of $\L^\an_n$ is 
that the former is naturally endowed with algebraic structures. 
Precisely, one easily checks that $\L_n$ is a $R$-module for the usual
operations (addition and scalar multiplication) on functions and that
the projection maps $\psi_n : \Omega \to \L_n$ and $\psi_{m,n} : \Omega_m
\to \Omega_n$ are all $R$-linear. 

The $\psi_{m,n}$'s are actually the exact algebraic analogue of the 
functions $\psi^\an_{m,n}$'s introduced in \S \ref{ssec:universe}. In 
order to state an analogue of Proposition \ref{prop:fibrepsi}, we 
introduce the additive group $G_{n+1}$ consisting of functions 
$\P^{d-1}(R_{n+1}) \to k^d$ and let it act on $\L_{n+1}$ by

$$\forall g_{n+1} \in \mathcal G_{n+1}, \quad
\forall f_{n+1} \in \mathcal L_{n+1}, \quad
g_{n+1} \bullet f_{n+1} = f_{n+1} + \pi^n g_{n+1}.$$

\begin{prop}
\label{prop:torsor}
The map $\psi_{n+1,n} : \L_{n+1} \to \L_n$ is surjective.
Moreover the action of $G_{n+1}$ stabilizes each fibre of $\psi_{n+1,n}$
and induces on it a free and transitive action.
\end{prop}

\begin{rem}
Recall that an action of a group $G$ over a space $X$ is free and
transitive if, given two any points $x,y \in X$, there always exists
a unique element $g \in G$ such that $y = gx$. This notably implies 
that, for all $x \in X$, the map $h_x : G \to X$, $g \mapsto gx$ is
a bijection. In particular $X$ is either empty or in bijection with 
$G$.
\end{rem}

\begin{proof}[Proof of Proposition \ref{prop:torsor}]
The surjectivity of $\psi_{n+1,n}$ comes from that of $\sp_{n+1,n}$
while the claimed properties on the action of $G_{n+1}$ are easily
checked.
\end{proof}

\begin{cor}
The set $\Omega_n$ has cardinality:
$$\card \Omega_n = q^{d \cdot \frac{q^d-1}{q-1} \cdot \frac{q^{n(d-1)}-1}{q^{d-1}-1}}.$$
\end{cor}

\begin{proof}
Proposition \ref{prop:torsor} implies:
$$\card \Omega_n = \card \Omega_{n-1} \cdot \card G_n =
\card \Omega_{n-1} \cdot q^{d \: \card \P^{d-1}(S_n)}.$$
The claimed formula follows by induction using Eq.~\eqref{eq:cardP1}.
\end{proof}

\begin{prop}
\label{prop:descOmega}
The mapping $\psi : f \mapsto (\psi_n(f))_{n \geq 1}$
induces a bijection between $\Omega$ and 
$\varprojlim_n \L_n$ where the latter is by definition the set 
of all sequences $(f_n)_{n \geq 1}$ with $f_n \in \L_n$ and 
$\psi_{n+1,n}(f_{n+1}) = f_n$ for all $n$.
\end{prop}

\begin{proof}
We define the inverse bijection of $\psi$. Let $(f_n)_{n \geq 1}$ be 
a sequence in $\varprojlim_n \L_n$. Let $a \in \P^{d-1}(K)$. The 
sequence of $f_n \circ \sp_n(a)$ defines an element in $\varprojlim_n
R_n^d$, \emph{i.e.} an element $f(a)$ in $R^d$ by completeness of $R$. 
This yields a function $f : \P^{d-1}(K) \to R^d$ making all the diagrams

\medskip

\noindent\hfill
\begin{tikzpicture}[xscale=3,yscale=1.5]
\node (P) at (0,0) { \ph $\P^{d-1}(K)$ };
\node (R) at (1,0) { \ph $R^d$ };
\node (Pn) at (0,-1) { \ph $\P^{d-1}(R_n)$ };
\node (Rn) at (1,-1) { \ph $R_n^d$ };
\draw[->] (P)--(R) node[above,midway,scale=0.8] { $f$ };
\draw[->] (P)--(Pn) node[left,midway,scale=0.8] { $\sp_n$ };
\draw[->] (R)--(Rn) node[right,midway,scale=0.8] { $p_n$ };
\draw[->] (Pn)--(Rn) node[above,midway,scale=0.8] { $f_n$ };
\end{tikzpicture}
\hfill\null

\smallskip

\noindent
commutative. One derives from this that $f$ is $1$-Lipschitz, 
\emph{i.e.} $f \in \Omega$. Moreover it is apparently an antecedent by
$\psi$ of the sequence $(f_n)_{n \geq 1}$. Finally, starting with $f \in
\Omega$, the above construction applied with $f_n = \psi_n(f)$ clearly
rebuilds $f$. This concludes the proof.
\end{proof}

\begin{prop}
\label{prop:uniformLn}
For all $n$ and all subset $E \subset \L_n$, we have:
$$\P[\psi_n(\omega) \in E] = \frac{\card E}{\card \L_n}.$$
In other words the map $\psi_n$ sends the probability measure on 
$\Omega$ to the uniform distribution on $\L_n$.
\end{prop}

\begin{proof}
Let $f_n, g_n \in \L_n$. Pick $h \in \Omega$ mapping to $g_n - f_n$ 
under $\psi_n$.
Taking advantage of the fact that $\psi_n$ is a group homomorphism, 
we derive that the translation by $h$ sends the fibre over $f_n$ to
the fibre over $g_n$. The properties of the Haar measure consequently
implies that all the fibres of $\psi_n$ have the same measure. The
proposition follows from this.
\end{proof}

\subsection{Reformulation of the main Theorem}
\label{ssec:statementalg}

We fix a positive integer $n$. 
Following the construction of Section \ref{sec:results}, given a 
function $f \in \L_n$, we define a Besikovitch set $N(f) \subset
R_n^d$ by:
$$N(f) = \bigcup_{a \in \P^{d-1}(S_n)} S_a(f)
\quad \text{with} \quad
S_a(f) = \big\{ t \cdot \can_n(a) + f(a) \::\: t\in R_n \big\}.$$
where we recall that $\can_n(a) \in R^d$ denote the unique representative
of $a$ whose first invertible coordinate is equal to $1$ (see \S 
\ref{ssec:projectivealg}). 
The relationship between the above construction and that of Section
\ref{sec:results} is made precise by the following lemma.

\begin{lem}
With the notations of \S \ref{ssec:statement}, we have:
$$X_n(f) = q^{-nd} \cdot \card N(\psi_n(f))$$
for all $f \in \Omega$.
\end{lem}

\begin{proof}
Set $f_n = \psi_n(f)$.
Proposition \ref{prop:neighbourhood} shows that:
$$X_n(f) = q^{-nd} \cdot \card p_n(N(f)).$$
It is then enough to show that $N(f_n)$ and $p_n(N(f))$ have
the same cardinality. We will actually show that these two sets are
equal.

Pick first $x \in N(f)$. Thus $x \in S_a(f)$ for some $a \in 
\P^{d-1}(K)$ from what we derive that $p_n(x) \in S_{\sp_n(a)}(f_n)$.
Therefore $p_n(x) \in N(f_n)$ and we have proved that $p_n(N(f))
\subset N(f_n)$. Conversely, take $x_n \in N(f_n)$, so that $x_n 
= t_n \cdot \can_n(a_n) + f_n(a)$ for some $a_n \in \P^{d-1}(S_n)$ and some
$t_n \in R_n$. Consider now $a \in \P^{d-1}(K)$ and $t \in R$ such that 
$\sp_n(a) = a_n$ and $p_n(t) = t_n$. Clearly $x = t\cdot\can_n(a) + f(a)$ sits
in $N(f)$ and, coming back to the definition of $\psi_n$ (\emph{cf}
Proposition \ref{prop:psin}), we observe that $p_n(x) = x_n$. 
Thus $x_n \in p_n(N(f))$ and we have proved the reverse inclusion.
\end{proof}

Combining the above lemma with Proposition \ref{prop:uniformLn},
we find that our main theorem can then be rephrased as follows.

\begin{theo}
\label{theo:mainalg}
Let $(u_n)_{n \geq 0}$ be the sequence defined by the recurrence:
$$\begin{array}{rcl}
u_0 = 1 & ; & \displaystyle
u_{n+1} = 1 - \left( 1 - \frac{u_n}{q^{d-1}} \right)^{q^{d-1}}.
\end{array}$$
and set:
$$u'_n = 1 - (1 - u_n)^{1 + q^{-1} + \cdots + q^{-(d-1)}}.$$
Then, for any position integer $n$:
$$\frac 1 {\card \L_n} \cdot \sum_{f_n \in \L_n} \card N(f_n)
= q^{nd} u'_n.$$
\end{theo}

\section{Proofs}
\label{sec:proofs}

In this section, we give complete proofs of Theorem \ref{theo:main} and 
Theorem \ref{theo:conj2} or, more precisely, of their algebraic 
analogues, namely Theorem \ref{theo:conj2alg} and Theorem 
\ref{theo:mainalg} respectively.
The strategy of the proof of Theorem~\ref{theo:conj2alg} follows closely 
that of the real case (see \cite[Theorem~2]{green} or 
\cite[Proposition~6.4]{babichenko}): a clever use of the Cauchy--Schwartz
inequality reduces the proof to finding good estimations of the size
of the intersections of two segments. This is achieved by
counting the number of solutions of some affine congruences.

The proof of Theorem~\ref{theo:mainalg} basically follows the same idea 
of understading the size of the intersections of unit length segments. 
Several complications nonetheless occur. The most significant one is 
that we cannot restrict ourselves to $2$ by $2$ intersections but need 
to study $s$ by $s$ intersections for any integer $s \geq 2$. Roughly
speaking, using the inclusion-exclusion principle, we will write $X_n$ 
as an alternating sum:
\begin{equation}
\label{eq:incexc}
X_n = X_{n,1} - X_{n,2} + X_{n,3} - \cdots + (-1)^s X_{n,s} + \cdots.
\end{equation}
We will then compute the mean of $X_{n,s}$ for all $s$, put it into
the above formula and end up this way with the value of $\E[X_n]$.
We would like to insist on the fact that, although $X_n$ is rather small
(at least less than $1$), the random variables $X_{n,s}$'s --- and their
mean --- may take very large values when $n$ is large. For instance
$\E[X_{n,2}]$ goes to infinity when $n$ grows up. There are then many
compensations and the miracle is that we will be able to keep 
\emph{exact} values during all the computation and then simplify the 
result.

\subsection{Kakeya conjecture in dimension $2$}
\label{ssec:conjKakeya}

We fix a positive integer $n$ and an integer $\ell \in \llbracket
0,n \rrbracket$. Let $B$ be a $\ell$-Besikovitch set in $R_n^2$. 
Our aim is to prove that:
\begin{equation}
\label{eq:conj2alg}
\card B \geq q^{2(n-\ell)} \cdot \frac 1 {\frac{q-1}{q+1} \: n + 1}
\end{equation}
By definition $B$ contains a segment $S_a$ of length $q^{-\ell}$ and
direction $a$ for 
each $a \in \P^1(S_n)$. Let $\psi_a$ be the indicator function of 
$S_a$. Set $\psi = \sum_{a \in \P^1(S_n)} \psi_a$. Note that $\psi$
vanishes outside $B$.
Applying the Cauchy--Schwarz inequality with $\psi$ and the indicator
function of $B$, we then get:
$$\Bigg(\sum_{x \in R_n^2} \psi(x)\Bigg)^2
 \leq \card B \cdot \sum_{x \in R_n^2} \psi(x)^2.$$
Noting that $\psi_a^2 = \psi_a$ and $\sum_{x \in R_n^2} \psi_a(x) = 
\card S_a$, the above inequality rewrites:
\begin{equation}
\label{eq:ineqconj2}
\Big(\sum_a \card S_a\Big)^2
 \leq \card B \cdot \sum_{a,b} \, \card (S_a \cap S_b)
\end{equation}
where $a$ and $b$ run over $\P^1(R_n)$. Recall that, given $a, b \in
\P^1(R_n)$, we have defined in \S \ref{ssec:projectivealg} an integer 
$v_n(a,b)$ between $0$ and $n$.

\begin{lem}
\begin{enumerate}[(a)]
\setlength\itemsep{0pt}
\item For $a \in \P^1(R_n)$, we have $\card S_a = q^{n-\ell}$.
\item For $a,b \in \P^1(R_n)$, we have $\card (S_a \cap S_b) \in
\{0, q^{\min(n-\ell, v_n(a,b))}\}$.
\end{enumerate}
\end{lem}

\begin{proof}
\emph{(a)}~Recall that $S_a$ consists of points $m_t = t \cdot \can_n(a) + a'$
where $t$ runs over $\pi^\ell R_n$ and $a' \in R_n^2$ is fixed. We claim that
these points are pairwise distinct. Indeed remember the $\piv_n(a)$-th
coordinate of $\can_n(a)$ is equal to $1$. Consequently the $\piv_n(a)$-th
coordinate of $m_t$ is $t+c$ where $c \in R_n$ is some constant. Our 
claim then becomes clear and it follows from it that the map $\pi^\ell 
R_n \to S_a$, $t \mapsto m_t$ is bijective. Hence $\card S_a = q^{n-\ell}$.

\smallskip

\noindent
\emph{(b)}~Thanks to what we have just explained, there exists $a', b' \in
R_n$ for which the cardinality of $S_a \cap S_b$ is equal to the number 
of solutions of the equation:
$$u \cdot \can_n(a) + a' = v \cdot \can_n(b) + b'$$
where the unknown are $u$ and $v$ and run over $\pi^\ell R_n$. The number of
solutions of this affine system is either $0$ or equal to the number
of solutions of the associated homogeneous system, namely:
$$\left(\begin{matrix} u & v \end{matrix}\right) \cdot
\left(\begin{matrix} a_1 & a_2 \\ b_1 & b_2 \end{matrix} \right) = 0$$
where $\can(a) = (a_1, a_2)$ and $\can(b) = (b_1, b_2)$. Thanks to 
(a direct adaptation of) Corollary~\ref{cor:distproj}, the above square 
matrix is equivalent to the diagonal matrix $\text{Diag}(1, 
\pi^{v_n(a,b)})$. In other words there exists a linear change of basis
$(u,v) \mapsto (u',v')$ after which our system rewrites:
$$\left\{ 
\begin{array}{l} u' = 0 \\ \pi^{v_n(a,b)} v' = 0 \end{array} \right.
\qquad \emph{i.e.} \qquad
\left\{ 
\begin{array}{l} u' = 0 \\ \pi^{n-v_n(a,b)} \text{ divides } v' 
\end{array} \right.$$
It is now clear that this system has $q^{\min(n-\ell,v_n(a,b))}$ 
solutions in $(\pi^\ell R_n)^2$.
\end{proof}

Coming back to the inequality~\eqref{eq:ineqconj2}, we obtain:
\begin{equation}
\label{eq:minBn}
\card B \geq 
 \frac{\big(\card \P^1(R_n) \cdot q^{n-\ell}\big)^2}
      {\sum_{a,b} q^{v_n(a,b)}} =
 \frac{q^{4n-2\ell-2} \cdot (q+1)^2}
      {\sum_{a,b} q^{v_n(a,b)}}
\end{equation}
where $a$ and $b$ run over $\P^1(R_n)$. 
Now fix $a \in \P^1(R_n)$ and observe that the set of $b$'s in 
$\P^1(R_n)$ for which $v_n(a,b) \geq v$ is a fibre of $\sp_{n,v}$.
Thanks to the results of \S \ref{ssec:projectivealg}, there are 
$q^{n-v}$ of them if $v > 0$ and, according to our convention, there
are $\card \P^1(R_n) = q^{n-1} (q+1)$ of them when $v = 0$.
Therefore, when $a$ remains fixed, we obtain:
\begin{align*}
\sum_b q^{v_n(a,b)} 
 & = 2 q^n + \sum_{v=1}^{n-1} q^v \cdot (q^{n-v} - q^{n-v-1}) \medskip \\
 & = (n+1) \: q^n - (n-1) \: q^{n-1} = n \: q^{n-1} (q-1) + q^{n-1}(q+1)
\end{align*}
Summing up over all $a$, we get:
\begin{align*}
\sum_{a,b} q^{v_n(a,b)} 
 & = \card \P^1(R_n) \cdot \big(n \: q^{n-1} (q-1) + q^{n-1}(q+1)\big) \\
 & = q^{n-1} (q+1) \cdot \big(n \: q^{n-1} (q-1) + q^{n-1}(q+1)\big)
\end{align*}
and injecting this in~\eqref{eq:minBn}, we end up with Eq.~\eqref{eq:conj2alg}
and the proof is complete.

\subsection{Average size of a random Kakeya set}
\label{ssec:incexc}

We now focus on the proof of Theorem~\ref{theo:mainalg} (which is
equivalent to Theorem~\ref{theo:main} thanks to the results of
\S \ref{ssec:statementalg}).
We fix a positive integer $n$ and endow $\Omega_n$ with the uniform
distribution. Recall that to any function $f \in \Omega_n$, we have
attached the Kakeya set:
$$N(f) = \bigcup_{a \in \P^{d-1}(S_n)} S_a(f)
\quad \text{with} \quad
S_a(f) = \big\{ t \cdot \can_n(a) + f(a) \::\: t\in R_n \big\}.$$
Set $C(f) = \card N(f)$ and, given in addition a subset $A$ of 
$\P^{d-1}(S_n)$, define:
\begin{equation}
\label{eq:CA}
C_A(f) = \card \bigcap_{a \in A} S_a(f).
\end{equation}
This defines a family of random variables on $\Omega_n$ and the
value we want to compute is the mean of $C$.
The inclusion-exclusion principle readily implies:
$$C(f) = \sum_{A \subset \P^{d-1}(S_n)} (-1)^{1 + \card A} \cdot C_A(f)$$
from what we get:
\begin{equation}
\label{eq:EC}
\E[C] = \sum_{A \subset \P^{d-1}(S_n)} (-1)^{1 + \card A} \cdot \E[C_A].
\end{equation}

\begin{rem}
The random variables $X_{n,s}$ considered in the introduction of the
Section \ref{sec:proofs} are related to the $C_A$'s as follows:
$$X_{n,s} = q^{-nd} \cdot 
\sum_{\substack{A \subset \P^{d-1}(S_n)\\ \card A = s}} C_A.$$
\end{rem}

Our strategy is now clear: first, we compute the expected values of the 
$C_A$'s and second, we inject the obtained result in Eq.~\eqref{eq:EC}. 
The first step is achieved in \S \ref{sssec:directional} while the 
second is reached in \S \ref{sssec:sumup}. The first paragraph (\S 
\ref{sssec:height}) is devoted to work out one important notion on 
which the rest of the proof will be based.

\subsubsection{The height function}
\label{sssec:height}

For $i \in \{1, \ldots, n\}$, choose and fix a total order on 
$\P^{d-1}(S_i)$ in such a way that the implication:
\begin{equation}
\label{eq:condorder}
a < b \quad \Longrightarrow \quad \sp_{i,i-1}(a) < \sp_{i,i-1}(b)
\end{equation}
holds for $i \geq 2$ and $a,b \in \P^{d-1}(S_i)$. (We
recall that the specialization maps $\sp_{i,i-1}$ were defined in \S
\ref{ssec:projectivealg}.) These orders can be built inductively on
$i$. Indeed first choose any total order on $\P^{d-1}(S_1)$. Then
choose any total order on each fibre of $\sp_{2,1}$ and glue them
together in order to build a total order on $\P^{d-1}(S_2)$ making
the implication~\eqref{eq:condorder} true for $i = 2$. Now continue 
this way with $i = 3, \ldots, n$.

\begin{deftn}
Let $A$ be a subset of $\P^{d-1}(S_n)$ of cardinality $\ell + 1$.
The \emph{height function} of $A$ is the function:
$$\begin{array}{rcl}
h_A : \quad \llbracket 1, \ell \rrbracket 
  & \to & \llbracket 1, n \rrbracket \smallskip \\
j & \mapsto & n - v_n(a_j, a_{j-1})
\end{array}$$
where the $a_j$'s ($0 \leq j \leq \ell$) are the elements of $A$
sorted by increasing order.
\end{deftn}

It is sometimes convenient to extend the function $h_A$ by setting
$h_A(0) = n$.
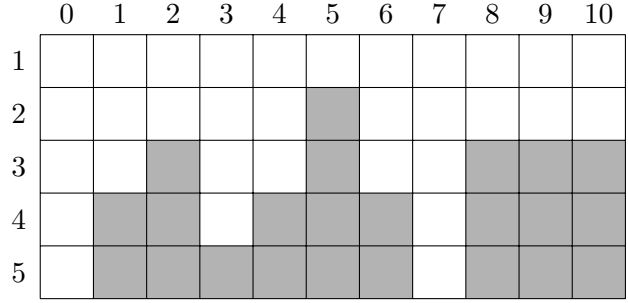
\begin{figure}
\noindent \hfill
\begin{tikzpicture}[scale=0.7,yscale=-1]
%\fill[black!60] (0,0) rectangle (1,5);
\begin{scope}[black!30]
\fill (1,3) rectangle (2,5);
\fill (2,2) rectangle (3,5);
\fill (3,4) rectangle (4,5);
\fill (4,3) rectangle (5,5);
\fill (5,1) rectangle (6,5);
\fill (6,3) rectangle (7,5);
\fill (7,5) rectangle (8,5);
\fill (8,2) rectangle (11,5);
\end{scope}
\draw (0,0) grid (11,5);
\node at (0.5,-0.4) { $0$ };
\node at (1.5,-0.4) { $1$ };
\node at (2.5,-0.4) { $2$ };
\node at (3.5,-0.4) { $3$ };
\node at (4.5,-0.4) { $4$ };
\node at (5.5,-0.4) { $5$ };
\node at (6.5,-0.4) { $6$ };
\node at (7.5,-0.4) { $7$ };
\node at (8.5,-0.4) { $8$ };
\node at (9.5,-0.4) { $9$ };
\node at (10.5,-0.4) { $10$ };
\node at (-0.4,0.5) { $1$ };
\node at (-0.4,1.5) { $2$ };
\node at (-0.4,2.5) { $3$ };
\node at (-0.4,3.5) { $4$ };
\node at (-0.4,4.5) { $5$ };
\end{tikzpicture}
\hfill\null

\caption{Representation of a height function (with $n = 5$ and $\ell = 10$)}
\label{fig:height}
\end{figure}
We will often represent a height function as a table with
$n$ rows (labeled from $1$ to $n$) and $\ell$ columns (labeled
from $1$ to $\ell$), where the cell $(i,j)$ is tinted in gray when $i
> h(j)$. Sometimes we will add a $0$-th column on the left with all
cells left white, in agreement with our convention $h_A(0) = n$.
Figure~\ref{fig:height} gives an example of such a representation.
It turns out that interesting informations can be read off immediately 
on this representation. For example, the numbers of white cells on the 
$i$-th row (including that on the $0$-th column) indicates the number of 
different values taken by the $\sp_{n,n+1-i}(a_j)$'s (for $0 \leq j \leq 
n$). More precisely, if $j < j'$, the equality $\sp_{n,n+1-i}(a_j) = 
\sp_{n,n+1-i}(a_{j'})$ holds if and only if the cells $(i,j+1), (i,j+2), 
\ldots, (i,j')$ are all left white.
This remark notably implies that:
\begin{equation}
\label{eq:distaj}
v_n(a_j, a_{j'}) = n - \max\big( h(j+1), h(j+2), \ldots, h(j') \big)
\end{equation}
provided that $j < j'$. 
In order to visualize even better the above properties, it can be 
helpful to fill the table of Figure~\ref{fig:height} by writing the
value $\sp_{n,n+1-i}(a_j)$ is the cell $(i,j)$. The three following
properties then hold:
\begin{enumerate}[(i)]
\setlength\itemsep{0pt}
\item \label{item:req1}
each cell contains an element which lies in the fibre of the 
element written just below (or, equivalently, each element of the
table specializes to the element written just above), 
\item \label{item:req2}
each gray cell contains the same element as the cell
immediately on the left,
\item \label{item:req3}
on each line, the elements are sorted in increasing order.
\end{enumerate}
Conversely remark that any filling of the table which satisfies the three
above requirements corresponds to one unique choice of $A$: it suffices
to read the $a_j$'s on the first line.
As we are going to explain now, this point of view will be particularly 
suitable for counting the number of subsets $A$ having a fixed height 
function.

\begin{deftn}
Let $h : \llbracket 1, \ell \rrbracket \to \llbracket 1, n \rrbracket$
be any function.

\noindent
The \emph{multiplicity function} of $h$ is the function $M(h) : 
\llbracket 1, \ell \rrbracket \to \N$
taking an integer $j \in \llbracket 1, \ell \rrbracket$ to the number
of indices $j' \in \llbracket 1, j \rrbracket$ for which:
$$h(j') = h(j) \quad \text{and} \quad
h(x) \leq h(j) \text{ for all } x \in \llbracket j, j' \rrbracket.$$
The \emph{weight function} of $h$ is the function $W(h) :
\llbracket 1, \ell \rrbracket \to \R$ defined by:
$$W(h)(j) = \frac 1{q^{d-1}} \cdot \frac {q^{d-1} - M(h)(j)}{M(h)(j)+1}.$$
The \emph{modified weight function} of $h$ is the function $W'(h) :
\llbracket 1, \ell \rrbracket \to \R$ defined by:
$$\begin{array}{r@{\hspace{0.5ex}}ll}
W'(h)(j) = 
& \displaystyle \frac 1{q^{d-1}} \cdot 
\frac {q^{d-1} - M(h)(j)}{M(h)(j)+1} & \text{if } h(j) \neq n \medskip \\
& \displaystyle \frac 1{q^{d-1}} \cdot 
\frac {1 + q + \cdots + q^{d-1} - M(h)(j)}{M(h)(j)+1} & \text{if } h(j) = n.
\end{array}$$
\end{deftn}

We emphasize that $j' = j$ is allowed in the definition of the 
multiplicity function, so that $M(h)(j)$ is always at least $1$. As an 
example, the values of the multiplicity function attached to the function 
$h_A$ represented on Figure~\ref{fig:height} are:
$${\renewcommand{\arraystretch}{1.2}
\begin{array}{|r||c|c|c|c|c|c|c|c|c|c|} 
\hline
j & 1 & 2 & 3 & 4 & 5 & 6 & 7 & 8 & 9 & 10 \\
\hline
h_A(j) & 3 & 2 & 4 & 3 & 1 & 3 & 5 & 2 & 2 & 2 \\
\hline
M(h_A)(j) & 1 & 1 & 1 & 1 & 1 & 2 & 1 & 1 & 2 & 3 \\
\hline
\end{array}}$$

\begin{prop}
\label{prop:heightcount}
Let $h : \llbracket 1, \ell \rrbracket \to \llbracket 1, n \rrbracket$
be a function.
The number of subset $A$ of $\P^{d-1}(S_n)$ (necessarily of cardinality 
$\ell+1$) whose height function is $h$ is:
\begin{equation}
\label{eq:heightcount}
\big(1 + q^{-1} + q^{-2} + \cdots + q^{-(d-1)}\big) \cdot q^{(d-1)n} \cdot
\prod_{j=1}^\ell \,\, W'(h)(j) \cdot q^{(d-1)\cdot h(j)}.
\end{equation}
\end{prop}

\begin{proof} 

Let us first explain that the value~\eqref{eq:heightcount} can be easily 
read off on the representation by cells (see Figure~\ref{fig:height}) we 
have introduced before. To do this, write $1 + q + \cdots + q^{d-1}$ 
in the cell $(0,n)$, write the number $q^{d-1} W'(h)(j)$ in the cell 
$(h(j),j)$ ($0 \leq j \leq \ell$) and $q^{d-1}$ in all other 
\emph{white} cells. In the example of Figure~\ref{fig:height}, we
get:

\medskip

\noindent \hfill
\begin{tikzpicture}[yscale=0.7,yscale=-1]
%\fill[black!60] (0,0) rectangle (1,5);
\begin{scope}[black!30]
\fill (1,3) rectangle (2,5);
\fill (2,2) rectangle (3,5);
\fill (3,4) rectangle (4,5);
\fill (4,3) rectangle (5,5);
\fill (5,1) rectangle (6,5);
\fill (6,3) rectangle (7,5);
\fill (7,5) rectangle (8,5);
\fill (8,2) rectangle (11,5);
\end{scope}
\draw (0,0) grid (11,5);
\node[scale=0.8] at (0.5,-0.4) { $0$ };
\node[scale=0.8] at (1.5,-0.4) { $1$ };
\node[scale=0.8] at (2.5,-0.4) { $2$ };
\node[scale=0.8] at (3.5,-0.4) { $3$ };
\node[scale=0.8] at (4.5,-0.4) { $4$ };
\node[scale=0.8] at (5.5,-0.4) { $5$ };
\node[scale=0.8] at (6.5,-0.4) { $6$ };
\node[scale=0.8] at (7.5,-0.4) { $7$ };
\node[scale=0.8] at (8.5,-0.4) { $8$ };
\node[scale=0.8] at (9.5,-0.4) { $9$ };
\node[scale=0.8] at (10.5,-0.4) { $10$ };
\node[scale=0.8] at (-0.4,0.5) { $1$ };
\node[scale=0.8] at (-0.4,1.5) { $2$ };
\node[scale=0.8] at (-0.4,2.5) { $3$ };
\node[scale=0.8] at (-0.4,3.5) { $4$ };
\node[scale=0.8] at (-0.4,4.5) { $5$ };
\node at (0.5,4.5) { $P$ };
\node at (0.5,3.5) { $A$ };
\node at (0.5,2.5) { $A$ };
\node at (0.5,1.5) { $A$ };
\node at (0.5,0.5) { $A$ };
\node at (1.5,2.5) { $\frac{A-1}2$ };
\node at (1.5,1.5) { $A$ };
\node at (1.5,0.5) { $A$ };
\node at (2.5,1.5) { $\frac{A-1}2$ };
\node at (2.5,0.5) { $A$ };
\node at (3.5,3.5) { $\frac{A-1}2$ };
\node at (3.5,2.5) { $A$ };
\node at (3.5,1.5) { $A$ };
\node at (3.5,0.5) { $A$ };
\node at (4.5,2.5) { $\frac{A-1}2$ };
\node at (4.5,1.5) { $A$ };
\node at (4.5,0.5) { $A$ };
\node at (5.5,0.5) { $\frac{A-1}2$ };
\node at (6.5,2.5) { $\frac{A-2}3$ };
\node at (6.5,1.5) { $A$ };
\node at (6.5,0.5) { $A$ };
\node at (7.5,4.5) { $\frac{P-1}2$ };
\node at (7.5,3.5) { $A$ };
\node at (7.5,2.5) { $A$ };
\node at (7.5,1.5) { $A$ };
\node at (7.5,0.5) { $A$ };
\node at (8.5,1.5) { $\frac{A-1}2$ };
\node at (8.5,0.5) { $A$ };
\node at (9.5,1.5) { $\frac{A-2}3$ };
\node at (9.5,0.5) { $A$ };
\node at (10.5,1.5) { $\frac{A-3}4$ };
\node at (10.5,0.5) { $A$ };
\end{tikzpicture}
\hfill\null

\medskip

\noindent
where we have set $A = q^{d-1}$ ($A$ for ``affine'') and 
$P = 1 + q + \cdots + q^{d-1}$ ($P$ for ``projective'').
It can then be easily checked that the quantity~\eqref{eq:heightcount} 
equals the product of all the numbers written in the above table.

Now recall that we have previously defined a bijection between the set 
of all $A$'s such that $h_A = h$ and the fillings of the table 
corresponding to $h$ obeying to the requirements 
\eqref{item:req1}--\eqref{item:req3} listed on page \pageref{item:req1}.
We are going to show that the number of such fillings of the $m$ last 
rows is exactly the product of the numbers appearing on the $m$ last 
rows. This will conclude the proof.
We proceed by induction on $m$. For $m = 1$, we have to count the 
number of strictly increasing sequences of elements of $\P^{d-1}(k)$ of 
length $c$ where $c$ is the number of white cells located on the last 
row. The data of such a sequence is obviously equivalent to the data of 
the set of its values. Since furthermore $\card \P^{d-1}(k) = P$, there 
are then $\binom P c$ such sequences and we are done for $m = 1$.
More generally, going from $m$ to $m+1$ is obtained in a similar fashion 
once we have noticed that the fibres of $\sp_{n-m+1, n-m}$ all have 
cardinality $A$ (see the discussion just below Eq.~\eqref{eq:cardP1}, 
page \pageref{eq:cardP1}).
\end{proof}

\subsubsection{Directional expected values}
\label{sssec:directional}

Throughout this paragraph, we fix a subset $A$ of $\P^{d-1}(S_n)$.
We write $A = \{a_0, a_1, \ldots, a_\ell\}$ with $a_0 < a_1 < \cdots
< a_\ell$ and denote by $h_A$ the height function of $A$.
Recall that we have defined a random variable $C_A$ on $\Omega_n$
by Eq.~\eqref{eq:CA}. The aim of this paragraph is to compute its mean.
In order to do so, we consider the following evaluation mapping:
$$\begin{array}{rcl}
\text{ev}_A : \quad \Omega_n & \to & (R_n^d)^{\ell+1} \smallskip\\
f & \mapsto & \big(f(a_0), f(a_1), \ldots, f(a_\ell)\big).
\end{array}$$
Clearly, $C_A(f)$ only depends on $\text{ev}_A(f)$ for $f \in \Omega_n$.
Moreover $\text{ev}_A$ is a group homomorphism, which notably implies 
that the fibres of $\text{ev}_A$ all have the same cardinality. As a 
consequence, letting $\calB_A$ denote the image of $\text{ev}_A$,
we get:
\begin{equation}
\label{eq:meanCA}
\E[C_A] = \frac 1{\card \calB_A} \cdot
\sum_{b \in \calB_A} \card \bigcap_{j=1}^\ell
\Sigma_{a_j}(b_j)
\end{equation}
where
$\Sigma_{a_j}(b_j) = \big\{ t \cdot \can_n(a_i) + b_i \::\: t\in R_n \big\}$.

\begin{lem}
The set $\mathcal B_A$ consists of tuples $(b_0, b_1, \ldots, 
b_{\ell+1}) \in (R_n^d)^{\ell+1}$ such that
$b_{j+1} \equiv b_j \pmod{\m^{n-h(j)}}$ for all $j \in \llbracket 1, 
\ell \rrbracket$.
\end{lem}

\begin{proof}
By definition of $h_A$, we have $v_n(a_j, a_{j+1}) = n-h_A(j)$ for all
$j$. Going back to the definition of $\Omega_n$, we deduce that, for 
any $f \in \Omega_n$ and $j \in \llbracket 1, \ell \rrbracket$, we
must have $f(a_{j+1}) \equiv f(a_j) \pmod{\m^{n-h_A(j)}}$.
In other words, $\text{ev}_A$ takes its values in $\mathcal B_A$.

Conversely pick $(b_0, b_1, \ldots, b_{\ell+1}) \in \mathcal B_A$. Given 
$a \in \P^{d-1}(S_n)$, let $j(a)$ be the smallest index for which 
$v_n(a,a_{j(a)})$ is maximal and set $f(a) = b_{j(a)}$. This defines a 
function $f : \P^{d-1}(S_n) \to R_n^d$ satisfying $f(a_j) = b_j$ for all 
$j$.
It remains to prove that $f \in \Omega_n$, \emph{i.e.} that $f$ 
is $1$-Lipschitz. Let $a, a' \in \P^{d-1}(S_n)$ and set for simplicity
$j = j(a)$ and $j' = j(a')$. Up to swapping $a$ and $a'$, we may assume
that $j \leq j'$. If $j = j'$ there is nothing to prove. Otherwise, it 
follows from Eq.~\eqref{eq:distaj} and the definition of $\calB_A$ that 
$b_j \equiv b_{j'} \pmod{\m^{v_n(a_j,a_{j'})}}$. 
This readily implies the $1$-Lipschitz condition under the extra 
assumption $v_n(a,a') \leq v_n(a_j,a_{j'})$ since then 
$\m^{v_n(a_j,a_{j'})} \subset \m^{v_n(a,a')}$. Let us now examine the
case where $v_n(a,a') > v_n(a_j,a_{j'})$. Put $\nu = v_n(a,a')$. 
From the assumption $v_n(a,a_j) \geq \nu$, we would derive:
$$v_n(a', a_{j'}) \geq v_n(a',a_j) \geq \min(v_n(a',a), v_n(a,a_j)) 
\geq \nu$$
and would deduce:
$$v_n(a_j,a_{j'}) \geq \min(v_n(a_j,a), v_n(a,a'), v_n(a',a_{j'})) =
\nu$$
which is a contradiction. Hence $v_n(a,a_j) < \nu$ and similarly
$v_n(a',a_{j'}) < \nu$. Noting that $v_n(x,z) = \min(v_n(x,y), v_n(y,z))$
as soon as $v_n(x,y) \neq v_n(y,z)$ (which comes from the very first
definition of $v_n$), we find:
$$v_n(a',a_j) = v_n(a,a_j) \geq v_n(a,a_{j'}) = v_n(a',a_{j'}).$$
Now we conclude by remarking that the above inequality cannot be true 
since it contradicts the minimality of $j'$ (remember that we had 
assumed $j < j'$).
\end{proof}

\begin{cor}
We have:
\begin{equation}
\label{eq:cardBA}
\card \calB_A = q^{nd} \cdot \prod_{j=1}^\ell q^{d \cdot h_A(j)}.
\end{equation}
\end{cor}

\begin{proof}
There are $q^{nd}$ possibilities for the choice of $b_0$. Once this
choice has been made, $b_1$ must satisfy $b_1 \equiv b_0 \pmod
{\m^{n-h_A(1)}}$, which leads to $q^{d \cdot h_A(1)}$ possibilities.
Repeating this reasoning, we end up with the announced formula.
\end{proof}

\begin{prop}
\label{prop:directionalmean}
We have:
$$\E[C_A] = q^n \cdot \prod_{j=1}^\ell \,\, q^{-(d-1) \cdot h_A(j)}.$$
\end{prop}

\begin{proof}
Fix a point $c \in R_n^d$. We are going to count the number of 
parameters $(b_0, b_1, \ldots, b_\ell) \in \calB_A$ for which $c$
lies on all lines $\Sigma_{a_j}(b_j)$ ($0 \leq j \leq \ell$). Call
$N_c$ this number.

We first focus on $b_0$. By definition $c \in \Sigma_{a_0}(b_0)$ if and 
only if there exists $t_0 \in R_n$ such that $t_0 \cdot \can(a_0) + b_0 
= c$. Since one of the coordinates of $\can(a_0)$ is equal to $1$, the 
mapping $t \mapsto t \cdot \can(a_0) + b_0$ is injective and there is 
then exactly $\card R_n = q^n$ acceptable values for $b_0$.

Suppose now that we are given $b_0, \ldots, b_j$ satisfying the
above condition and let us count the number of possibilities for
completing the sequence with an extra term $b_{j+1}$. This $b_{j+1}$
has to satisfy the two following conditions:
\begin{align*}
\exists t_{j+1} \in R_n, \quad 
& t_{j+1} \cdot \can(a_{j+1}) + b_{j+1} = c \\
& b_{j+1} \equiv b_j \pmod{\m^{n-h_A(j)}}
\end{align*}
Our problem then amounts to counting the number of values $t_{j+1}
\in R_n$ such that:
\begin{equation}
\label{eq:keycongr}
t_{j+1} \cdot \can(a_{j+1}) + b_j \equiv c \pmod{\m^{n-h_A(j)}}.
\end{equation}
Since $c \in \Sigma_{a_j}(b_j)$, we know that there exists some $t_j
\in R_n$ such that $t_j \cdot \can(a_j) + b_j = c$. By 
Proposition~\ref{prop:distn}, we know moreover that $\can(a_j) \equiv \can
(a_{j+1}) \pmod{\m^{n-h_A(j)}}$. Thus $t_{j+1} = t_j$ is a solution
of~\eqref{eq:keycongr} and, using again that $\can(a_{j+1})$ has one
coordinate equal to $1$, we find that Eq.~\eqref{eq:keycongr} rewrites 
$t_{j+1} \equiv t_j \pmod{\m^{n-h_A(j)}}$. There are thus $q^{h_A(j)}$
possibilities for $t_{j+1}$.

As a consequence of the previous discussion, we find that 
$N_c = q^n \cdot q^{h_A(1)} \cdot q^{h_A(2)} \cdots q^{h_A(\ell)}$
(independantly on $c$). Finally notice that:
$$\sum_{b \in \calB_A} \card \bigcap_{j=1}^\ell
\Sigma_{a_j}(b_j) = \sum_{c \in R_n^d} N_c = q^{nd} \cdot q^n \cdot q^{h_A(1)} \cdot q^{h_A(2)} \cdots q^{h_A(\ell)}$$
and conclude by injecting this equality together with Eq.~\eqref{eq:cardBA} 
in Eq.~\eqref{eq:meanCA}.
\end{proof}

\subsubsection{Summing up all contributions}
\label{sssec:sumup}

Let $\mathcal H_n$ be the set of all functions $h : \llbracket 1, \ell 
\rrbracket \to \llbracket 1, n \rrbracket$ for $\ell$ varying in 
$\llbracket 0, +\infty \llbracket$ (agreeing as usual that there exists
a unique function $h : \emptyset \to \llbracket 1, n \rrbracket$).
For $h \in \mathcal H_n$, denote $\ell(h)$ its $\ell$. Combining 
Proposition~\ref{prop:heightcount} and 
Proposition~\ref{prop:directionalmean}, we find that the expected value 
of $C$ is:
\begin{equation}
\label{eq:EC2}
\E[C] = 
\big(1 + q^{-1} + q^{-2} + \cdots + q^{-(d-1)}\big) \cdot q^{nd}
\cdot \sum_{h \in \mathcal H_n} (-1)^{\ell(h)} \prod_{i=1}^{\ell(h)} \,\, W'(h)(i)
\end{equation}
Recall that we have defined a sequence $(u_n)_{n \geq 0}$ by:
\begin{equation}
\label{eq:recun}
u_0 = 1 \quad ; \quad
u_n = 1 - \left( 1 - \frac{u_{n-1}}{q^{d-1}} \right)^{q^{d-1}}.
\end{equation}

\begin{prop}
The following formula holds:
$$u_n = \sum_{h \in \mathcal H_n} (-1)^{\ell(h)} 
\prod_{i=1}^{\ell(h)} \,\, W(h)(i).$$
\end{prop}

\begin{proof}
For simplicity, we set $w(h) = \prod_{i=1}^{\ell(h)} W(h)(i)$.
The key observation is the following: to each $h \in \mathcal H_n$, one
can attach a finite sequence $h_0, h_1, \ldots, h_m$ of functions in
$\mathcal H_{n-1}$ as follows. Let $j_1 < j_2 < \cdots < j_m$ be the 
integers for which $h(j_i) = n$, set $j_0 = 0$ and $j_{m+1} = \ell(h)+1$
and, for $i \in \llbracket 0, m \rrbracket$, define:
$$\begin{array}{rcl}
h_i : \llbracket 1, j_{i+1}{-}j_i{-}1 \rrbracket 
& \to & \llbracket 1, n{-}1 \rrbracket \smallskip \\
j & \mapsto & h(j+j_i).
\end{array}$$
On the representation of Figure~\ref{fig:height}, the functions $h_i$'s
then correspond to the bands (with last row erased) located between two white columns.
This construction clearly defines
a bijection between $\mathcal H_n$ and the set of finite sequences of
elements of $\mathcal H_{n-1}$. This bijection is moreover compatible
with the length and the weight functions in the following sense: if $h$ 
corresponds to $(h_0, h_1, \ldots, h_m)$ then
$\ell(h) = m + \ell(h_1) + \ell(h_2) + \cdots + \ell(h_m)$ and
$$\begin{array}{r@{\hspace{0.5ex}}ll}
W(h)(j) & = W(h_i)(j-j_i) & \text{for } j_i < j < j_{i+1} \medskip \\
W(h)(j_i) & = \displaystyle \frac 1{q^{d-1}} \cdot \frac {q^{d-1}-i} {i+1} 
\end{array}$$
Hence $w(h) = q^{-(d-1)(m+1)} \cdot \binom{q^{d-1}}{m+1} \cdot w(h_1) \cdot w(h_2) 
\cdots w(h_m)$. Taking the sum over all $h \in \mathcal H_n$, we
find the relation:
\begin{align*}
\sum_{h \in \mathcal H_n} (-1)^{\ell(h)} w(h) 
& = \sum_{m=0}^\infty (-1)^m \cdot 
\binom{q^{d-1}}{m+1} \cdot \Big(\frac 1{q^{d-1}}\Big)^{m+1}
\sum_{\substack{h_0, \ldots, h_m\\ \in \mathcal H_{n-1}}} \,\,
\prod_{i=0}^m (-1)^{\ell(h_i)} w(h_i) \\
& = \sum_{m=0}^\infty (-1)^m \cdot 
\binom{q^{d-1}}{m+1} \cdot \Big(\frac 1{q^{d-1}}\Big)^{m+1} 
\cdot \Bigg(\sum_{h \in \mathcal H_{n-1}} (-1)^{\ell(h)} w(h) \Bigg)^{\!m+1} \\
& = 1 \,\,-\,\, \sum_{m'=1}^{q^{d-1}} \binom{q^{d-1}}{m'} \cdot \Big(-\frac 1{q^{d-1}}\Big)^{\!m'} 
\cdot \Bigg(\sum_{h \in \mathcal H_{n-1}} (-1)^{\ell(h)} w(h) \Bigg)^{m'} \\
& = 1 \,\,- \,\, \Bigg(1 - \frac 1{q^{d-1}}
\sum_{h \in \mathcal H_{n-1}} (-1)^{\ell(h)} w(h) \Bigg)^{q^{d-1}}.
\end{align*}
The proposition now follows by comparing the above relation with
Eq.~\eqref{eq:recun}.
\end{proof}

Slightly adapting the arguments of the above proof, we get:
$$\begin{array}{l}
\displaystyle
\big(1 + q^{-1} + q^{-2} + \cdots + q^{-(d-1)}\big)
\cdot \sum_{h \in \mathcal H_n} (-1)^{\ell(h)} \prod_{i=1}^{\ell(h)} \,\, W'(h)(i) \medskip \\
\hspace{17em} \displaystyle
= 1 - \left( 1 - \frac{u_{n-1}}{q^{d-1}} \right)^{1 + q + \cdots + q^{d-1}} = u'_n
\end{array}$$
where $u'_n$ is defined in the statement of Theorem~\ref{theo:mainalg}
(page \pageref{theo:mainalg}).
Using Eq.~\eqref{eq:EC2}, we end up with $\E[C] = q^{nd} u'_n$ and
Theorem~\ref{theo:mainalg} is proved.

\section{Numerical simulations}
\label{sec:numeric}

We recall that the main objects studied in this paper are the random 
Kakeya sets and especially the random variables $X_n$ (defined in \S 
\ref{ssec:statement}) that measure their size. In this last section, We 
present several numerical simulations showing the behaviour of the 
$X_n$'s beyond their mean. 

All our experiments have been done over the field of $2$-adic numbers 
$\Q_2$. We recall briefly that $\Q_2$ is the completion of $\Q$ for 
the $2$-adic norm $|\cdot|_2$ defined, for two integers $n$ and $m$, by:
$$\begin{array}{cl}
& |n|_2 = 2^{-v}
\quad \text{if $2^v$ is the highest power of $2$ dividing $n$} \smallskip \\
\text{and} & \big|\frac n m\big|_2 = \frac{|n|_2}{|m|_2}.
\end{array}$$
The unit ball of $\Q_2$ is the so-called ring of $2$-adic integers $\Z_2$.
Any element $x$ in it can be uniquely written as a convergent series
$$x = s_0 + 2 s_1 + 2^2 s_2 + 2^3 s_3 + \cdots + 2^n s_n + \cdots$$
where the $s_i$'s all lie in $S = \{0, 1\}$ (decomposition in $2$-basis).
The $s_i$'s define mutually independent Bernoulli variables of 
parameter $\frac 1 2$ on $\Z_2$. In other words, generating a random 
element in $\Z_2$ reduces to pick each digit $s_i$ uniformly in $S$ and 
independently.

\subsection{Empirical distribution of the variables $X_n$}

We recall that our universe $\Omega$ is the set of $1$-Lipschitz 
functions $\P^{d-1}(K) \to R^d$. By the results of \S 
\ref{ssec:universe}, $\Omega$ comes equipped with projection maps 
$\Omega \to \Omega^\an_n$ where $\Omega^\an_n$ was defined as the subset 
of $\Omega$ consisting of functions which are constant on each closed 
ball of radius $q^{-n}$ and takes their values in $\llbracket 0, 2^n{-}1 
\rrbracket^d$. Alternatively functions in $\Omega^\an_n$ can be viewed as 
mapping $\P^{d-1}(S_n) \to S^d$ satisfying an extra condition (see \S
\ref{ssec:universealg}). Two other interesting features of $\Omega^\an_n$
are the following: (1)~the measure induces on $\Omega^\an_n$ by the 
projection $\Omega \to \Omega^\an_n$ is the uniform distribution and
(2)~the random variable $X_n$ factors through $\Omega^\an_n$.

We recall also that one can furthermore decompose any function in
$\Omega^\an_n$ as a sum:
\begin{equation}
\label{eq:decompfn}
(g_1 \circ \sp_1) + 2 \cdot (g_2 \circ \sp_2)
+ 2^2 \cdot (g_3 \circ \sp_3) + \cdots + 2^{n-1} \cdot (g_n \circ \sp_n)
\end{equation}
where $g_i : \P^{d-1}(S_i) \to S^d$ is any function and conversely that
any function of the shape~\eqref{eq:decompfn} lies in $\Omega^\an_n$. 
Generating a random function in $\Omega^\an_n$ then reduces to pick
the $g_i$'s ($1 \leq i \leq n$) uniformly and independently.
Picking each $g_i$ is also easy: we
enumerate the elements of $\P^{d-1}(S_i)$ (this can be done using the
results of \S \ref{ssec:projectivealg}) and choose their image randomly
and independantly in $S^d$.
\begin{figure}
\begin{lstlisting}[language=Python,basicstyle=\footnotesize,keywordstyle=\bf\color{blue},otherkeywords={:},commentstyle=\color{darkgreen}\it,texcl=true]
def random_lipschitz_iter(d,n):
    if n == 1:
        # Run over elements $a \in \P^{d-1}(\Z/2\Z)$ according to the position of the first nonzero coordinate
        for piv in range(d):
            for a in xmrange_iter(piv*[[0]] + [[1]] + (d-1-piv)*[[0,1]]):
                # Gererate a random image $b \in (\Z/2\Z)^d$ of $a$
                b = [ randint(0,1) for _ in range(d) ]
                yield(piv, vector(a), vector(b))
    else:
        # Run over elements $a \in \P^{d-1}(\Z/2^{n-1}\Z)$ and call $b$ the image of $a$
        for (piv,a,b) in random_lipschitz_iter(d,n-1):
            q = 2**(n-1)
            # Run over the elements $a+a'$ of the fibre of $\sp_{n,n-1}$ above $a$
            for aprime in xmrange_iter(piv*[[0,q]] + [[0]] + (d-1-piv)*[[0,q]]):
                # Generate a random image $b + b' \in (\Z/2\Z)^d$ (with $2^{n-1}$ divides $b'$) of $a+a'$
                bprime = [ q*randint(0,1) for _ in range(d) ]
                yield(piv, a+vector(aprime), b+vector(bprime))
\end{lstlisting}
\caption{\textsc{SageMath} function generating a random element
in $\Omega^\an_n$ for $K = \Q_2$}
\label{fig:sage}
\end{figure}
The \textsc{SageMath} function presented in Figure~\ref{fig:sage}
generates a random element $f_n \in \Omega^\an_n$ according to the
uniform distribution. More precisely, it returns an iterator over
the sequence of triples $(\piv(a), a, f_n(a))$ where $a$ runs over
$\P^{d-1}(S_n)$. (Note that the first coordinate $\piv(a)$ is useful 
for the recursion but may be then omitted.) One nice feature of this
implementation is its memory cost which (almost) does not grow with $n$.

\begin{figure}
\noindent \hfill
{\renewcommand{\arraystretch}{1.3}
\begin{tabular}{|c||c|c|c|c|c|c|c|}
\hline
$n$ & 5 & 6 & 7 & 8 & 9 & 10 & 11 \\
\hline
\begin{tabular}{@{}c@{}}
$\E[X_n]$ \vspace{-0.5ex} \\ (theoretical value)
\end{tabular}
  & 0.534 & 0.487 & 0.448 & 0.415 & 0.386 & 0.362 & 0.340 \\
\hline
\begin{tabular}{@{}c@{}}
$\E[X_n]$ \vspace{-0.5ex} \\ (empirical value)
\end{tabular}
  & 0.534 & 0.487 & 0.448 & 0.415 & 0.386 & 0.362 & 0.340 \\
\hline
\begin{tabular}{@{}c@{}}
$\sigma[X_n]$ \vspace{-0.5ex} \\ (empirical value)
\end{tabular}
  & 0.0316 & 0.0229 & 0.0169 & 0.0126 & 0.0097 & 0.0076 & 0.0061 \\
\hline
\end{tabular}}
\hfill\null

\caption{Expected value and standard deviation of $X_n$ for $K = \Q_2$ 
and $d=2$}
\label{fig:Xn2}
\end{figure}

\begin{figure}
\noindent \hfill
{\renewcommand{\arraystretch}{1.3}
\begin{tabular}{|c||c|c|c|c|c|c|c|}
\hline
$n$ & 3 & 4 & 5 & 6 & 7 & 8 & 9 \\
\hline
\begin{tabular}{@{}c@{}}
$\E[X_n]$ \vspace{-0.5ex} \\ (theoretical value)
\end{tabular}
  & 0.628 & 0.551 & 0.490 & 0.442 & 0.402 & 0.369 & 0.341 \\
\hline
\begin{tabular}{@{}c@{}}
$\E[X_n]$ \vspace{-0.5ex} \\ (empirical value)
\end{tabular}
  & 0.628 & 0.551 & 0.490 & 0.442 & 0.402 & 0.369 & 0.341 \\
\hline
\begin{tabular}{@{}c@{}}
$\sigma[X_n]$ \vspace{-0.5ex} \\ (empirical value)
\end{tabular}
  & 0.0502 & 0.0371 & 0.0286 & 0.0227 & 0.0187 & 0.0155 & 0.0132 \\
\hline
\end{tabular}}
\hfill\null

\caption{Expected value and standard deviation of $X_n$ for $K = \Q_2$ 
and $d=3$}
\label{fig:Xn3}
\end{figure}

The tables of Figure~\ref{fig:Xn2} (page \pageref{fig:Xn2}) and 
Figure~\ref{fig:Xn3} (page \pageref{fig:Xn3}) show the 
expected value and the standard deviation of some of $X_n$'s observed on 
a sample (renewed for each value of $n$) of $100,000$ random Kakeya sets 
in dimension $2$ and $3$ respectively.
We note in particular that:

\vspace{-1.5ex}

\begin{itemize}
\setlength\itemsep{0pt}
\item the empirical mean agrees with the theoretical one (given by 
Theorem \ref{theo:main}) up to $10^{-3}$,
\item the standard deviation is quite small and seems to converge to 
$0$ faster than the mean, \emph{i.e.} faster than $\frac 1 n$ (although
this phenomenon is less apparent in dimension $3$).
\end{itemize}

\vspace{-1ex}

\begin{figure}
\noindent\hfill
\begin{tabular}{m{2cm}m{10cm}}
$n = 5$: & \includegraphics[width=10cm]{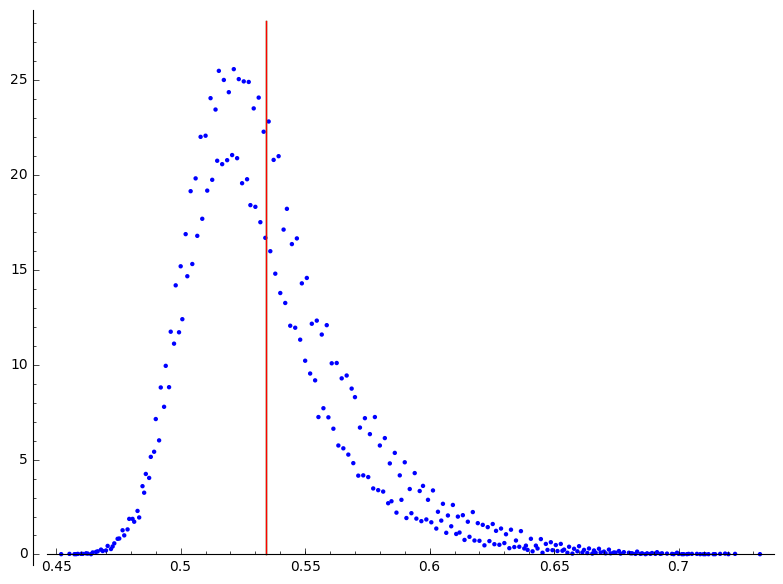} \\
$n = 10$: & \includegraphics[width=10cm]{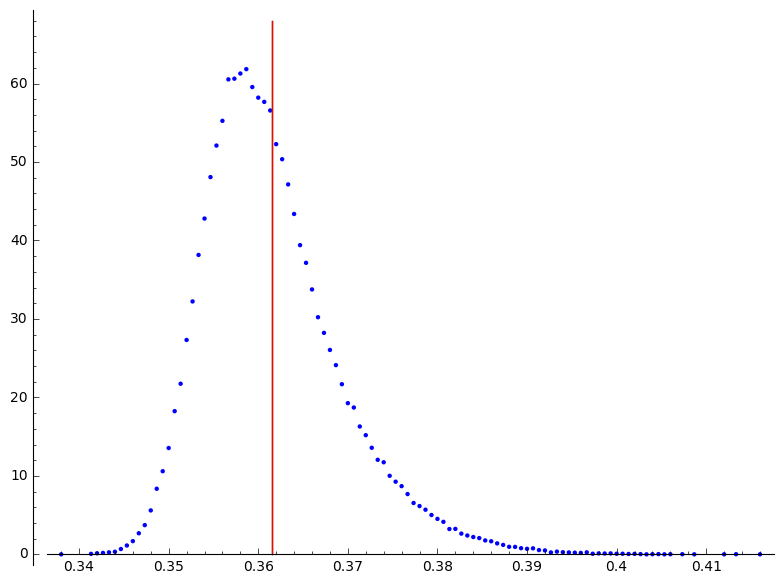} \\
$n = 11$: & \includegraphics[width=10cm]{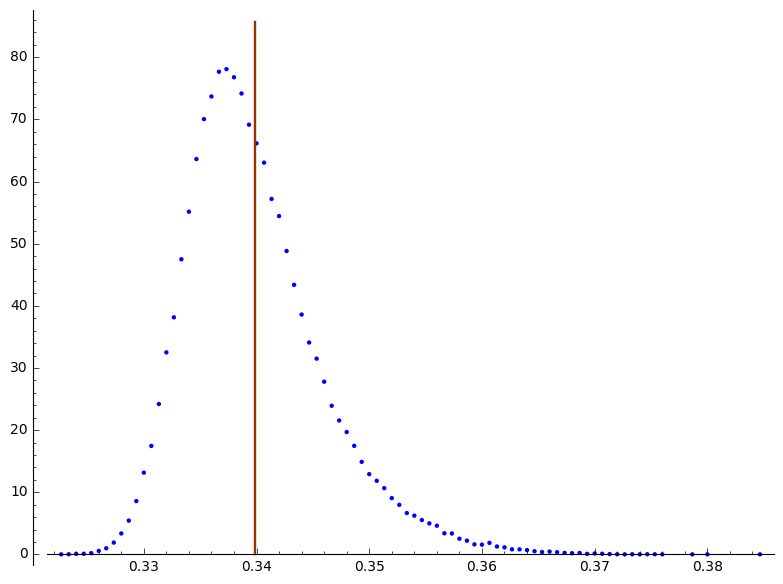}
\end{tabular}
\hfill\null

\bigskip

\caption{Empirical density of $X_n$ for $K = \Q_2$ and $d=2$}
\label{fig:density2}
\end{figure}

\begin{figure}
\noindent\hfill
\begin{tabular}{m{2cm}m{10cm}}
$n = 3$: & \includegraphics[width=10cm]{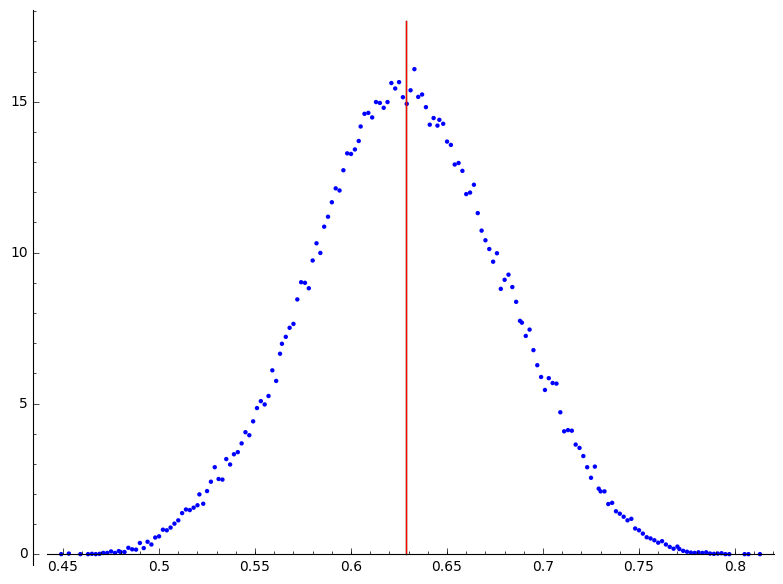} \\
$n = 8$: & \includegraphics[width=10cm]{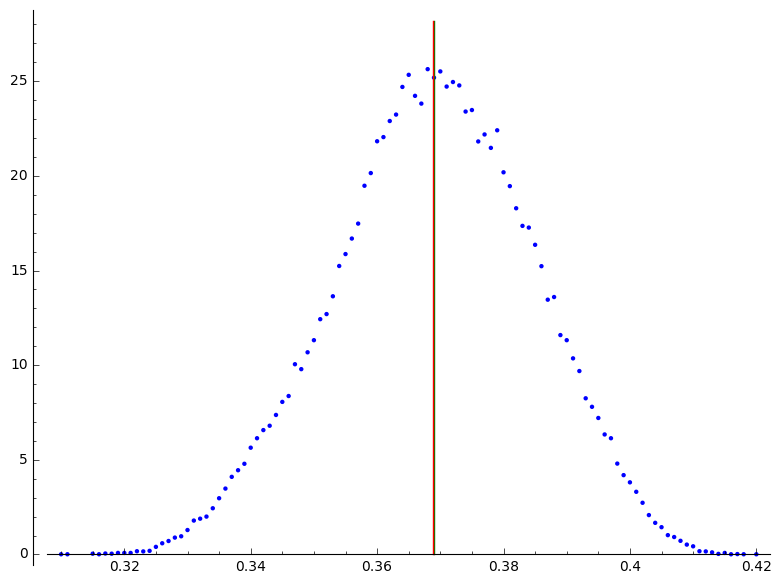} \\
$n = 9$: & \includegraphics[width=10cm]{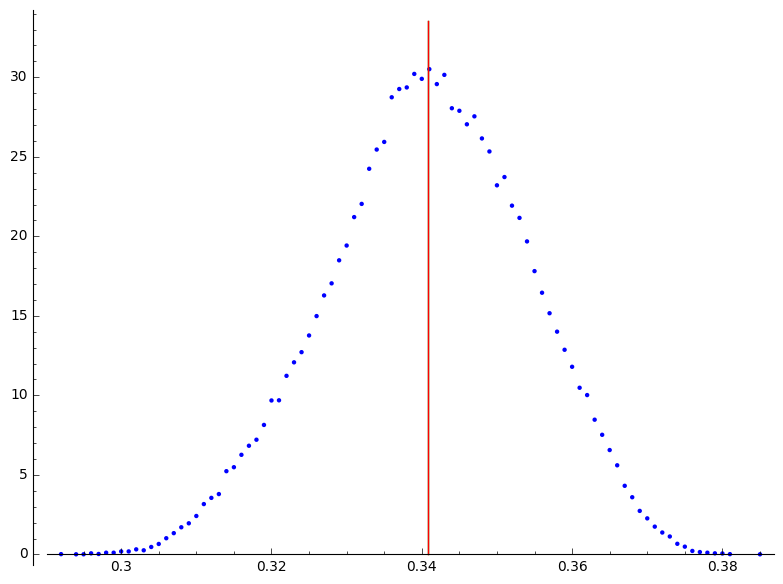}
\end{tabular}
\hfill\null

\bigskip

\caption{Empirical density of $X_n$ for $K = \Q_2$ and $d=3$}
\label{fig:density3}
\end{figure}

Going further one can draw the empirical ``density''\footnote{It is not 
actually a density in the usual sense because the variables $X_n$'s take 
their values in a \emph{discrete} subset of $\R$.}: we subdivise $\R$ 
into small intervals and count, for each of them, the proportion of 
sample points (renormalized by the size of the interval) leading to a 
point in it. The results are displayed in Figure~\ref{fig:density2} 
(page \pageref{fig:density2}) and Figure~\ref{fig:density3} (page 
\pageref{fig:density3}) in dimension $2$ and $3$ respectively. The red 
and green vertical lines (which actually always collapse) in these 
pictures indicate the theoretical mean and the empirical mean of $X_n$ 
respectively.

For a fixed dimension, the density curves (for various $n$) all have a
similar shape. This may suggest that the law of $X_n$ --- correctly
renormalized --- converges to some limit. We believe that it would
be very interesting to investigate further this question. For example
if one can compute this limit and check that it is zero until some
point, it would eventually imply the Kakeya conjecture for almost all
non-archimedean Kakeya sets.

We finally remark that, on the first diagram of Figure~\ref{fig:density2}, 
one can clearly separate two curves. This reflects a parity phenomenon: 
$q^{nd} X_n = 2^{10} X_5$ is even with probability $\approx 73\%$ and 
odd with probability $\approx 27\%$.
The curve below then corresponds to odd values of $X_5$ while the curve 
above corresponds to even values. This phenomenon tends to disappear
rapidly when $n$ grows up.

\subsection{Visualizing a random $2$-adic Kakeya set}

In order to draw a $2$-adic Kakeya set sitting naturally in $\Z_2^d$, 
we will necessarily need to relate $\Z_2$ and $\R$. In order to do so,
we use the ``reverse'' function $r : \Z_2 \to [0,1]$ mapping the 
$2$-adic integer $\sum_{i=0}^\infty 2^i s_i$ (with $s_i \in \{0,1\})$
to the real number $\sum_{i=0}^\infty 2^{-i-1} s_i$. 

Note that $r$ is continuous (it is actually $1$-Lipschitz) but not 
injective since the binary representation of a real number fails to be 
unique in general. For instance $\frac 1 2$ has two preimages which are 
$1 \in \Z_2$ and $-2 \in \Z_2$. The closed intervals $[0,\frac 1 2]$ and 
$[\frac 1 2, 1]$ correspond to the disjoint cosets $2 \Z_2$ and $2 \Z_2 
+ 1$ respectively. Note that the latter are open and closed in $\Z_2$. 
More generally all real number of the form $\frac a{2^n}$ have two 
distinct preimages in $\Z_2$ and there always exist two closed interval 
meeting $\frac a{2^n}$ corresponding to two open closed subsets of 
$\Z_2$.

\begin{rem}
There actually exist closed embeddings $\Z_2 \to \R$; an example of it 
is the Cantor mapping $C$ taking $\sum_{i=0}^\infty 2^i s_i \in \Z_2$ to 
$2 \cdot \sum_{i=0}^\infty 3^{-i-1} s_i$. The image of $C$ is the 
usual triadic Cantor set and $C$ induces an homeomorphism between
it and $\Z_2$. We nevertheless preferred to use $r$ because it maps
$\Z_2$ to an interval whereas $C$ maps $\Z_2$ to a null set. Working
with $C$ has then two disadvantages: it would lead to undrawable 
pictures on the one hand and would not reflect properly the properties 
we want to emphasize on the other hand.
\end{rem}

\begin{figure}
\noindent \hfill
\includegraphics[width=10cm]{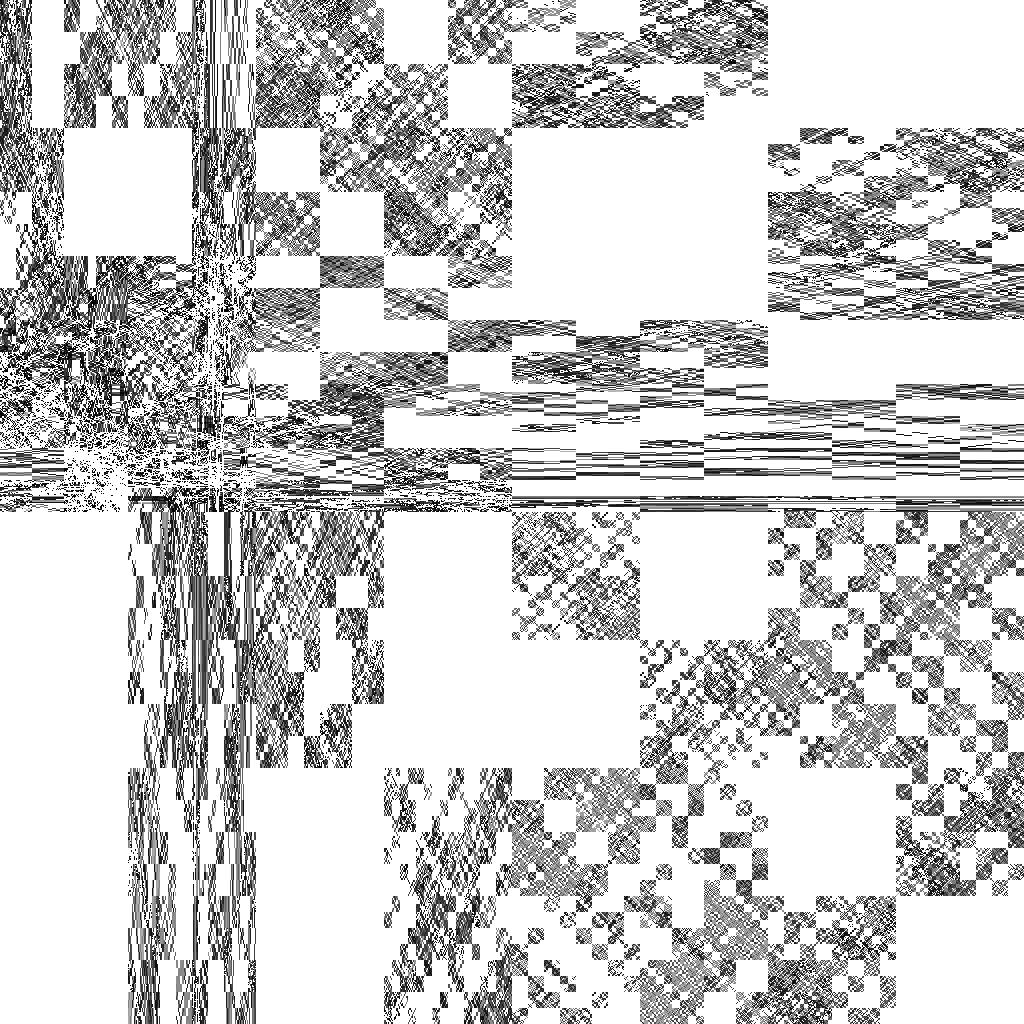}
\hfill \null

\bigskip

\caption{A $2$-dimensional random Kakeya set over $\Q_2$}
\label{fig:kakeya2}
\end{figure}

\begin{figure}
\noindent \hfill
\includegraphics[width=13cm]{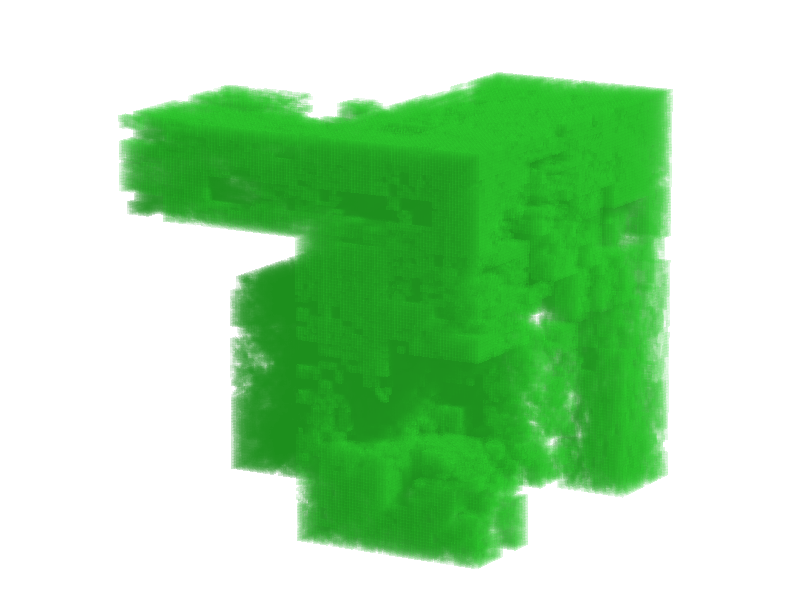}
\hfill \null

\bigskip

\caption{A $3$-dimensional random Kakeya set over $\Q_2$}
\label{fig:kakeya3}
\end{figure}

Viewing $\Z_2^2$ in $\R^2$ through the map $(r,r)$, the picture of 
Figure~\ref{fig:kakeya2} (page \pageref{fig:kakeya2}) represents a 
random Kakeya set --- or more precisely its $(2^{-13})$-neighbourhood --- 
in $\Z_2^2$. An animation showing a $2$-adic needle moving continuously 
in the $2$-adic plane and filling a $2$-adic Kakeya set is available at 
the URL: 
\begin{center} 
\url{http://xavier.toonywood.org/papers/publis/kakeya/kakeya-2d.gif} 
\end{center}
Finally, a $3$-dimensional $2$-adic Kakeya set is displayed on
Figure \ref{fig:kakeya3} and a movie showing it on different angles can
be found at:
\begin{center} 
\url{http://xavier.toonywood.org/papers/publis/kakeya/kakeya-3d.mp4}
\end{center}

\appendix
\section{Appendix: Discrete valuation fields}
\label{app:DVF}

This appendix is dedicated to readers who are not familiar with 
non-archimedean geometry. It presents a quick summary of the most
important basic definitions and facts of the domain.
All the material presented below is very classical.

\subsection*{Definitions}

A \emph{discrete valuation field} is a field $K$ equipped with a map 
$\val : K \to \Z \cup \{+\infty\}$ (the so-called \emph{valuation}) 
satisfying the following axioms:
\begin{enumerate}[(i)]
\setlength\itemsep{0pt}
\item $\val(x) = +\infty$ if and only if $x = 0$,
\item $\val(xy) = \val(x) + \val(y)$,
\item $\val(x+y) \geq \min(\val(x), \val(y))$
\end{enumerate}
for all $x$ and $y$ in $K$. The valuation $\val$ is \emph{non trivial} 
if there exists an element $x \in K^\star$ with $\val(x) \neq 0$. Under
this additional assumption, the set $\val(K^\star)$ is a subgroup of
$\Z$ and therefore is equal to $n\Z$ for some positive integer $n$.
An element $\pi \in K$ of valuation $n$ is called a \emph{uniformizer} 
of $K$. One can always renormalize the valuation (by dividing it by $n$) 
in order to ensure $n=1$.

\medskip

The valuation on $K$ readily defines a family of absolute values 
$|{\cdot}|_a$ ($a > 1$) on $K$ by:
$$\forall a \in (1,\infty), \, \forall x \in K, \quad
|x|_a = a^{-\val(x)}$$
with the convention that $a^{-\infty} = 0$. Each of these absolute 
values defines a distance $d_a$ on $K$ by the usual formula $d_a(x,y) 
= |x-y|_a$. It is easily seen that all these distances define the same 
topology on $K$. We underline that $d_a$ is ultrametric in the sense
that:
\begin{equation}
\label{eq:ultrametric}
\forall x, y, z \in K, \qquad 
d_a(x,z) \leq \max\big(d_a(x,y), d_a(y,z)\big).
\end{equation}
This stronger version of the triangle inequalities has unexpected and
important consequences. 
For instance it implies that $d_a(x,z) = \max\big(d_a(x,y), 
d_a(y,z)\big)$ as soon as $d_a(x,y) \neq d_a(y,z)$, showing then
that every triangle in $K$ is isosceles. Similarly if two balls 
$B_1$ and $B_2$ of $K$ meet, we necessarily have $B_1 \subset B_2$
or $B_2 \subset B_1$.

\medskip

Let $R$ be the closed unit ball of $K$ (this does not depend on the 
parameter $a$); alternatively $R$ is the subset of $K$ consisting of 
elements $x$ with nonnegative valuation. An important remark following 
from axioms~(ii) and~(iii) is that $R$ is a subring of $K$; it is 
usually called the \emph{ring of integers} of $K$. The invertible
elements in $R$ are clearly exactly the elements of norm $1$ (since 
the norm is multiplicative). On the contrary, the open unit ball
$\m$ is an ideal of $R$. It is actually the
unique maximal ideal of $R$ (showing that $R$ is a local ring). It
is moreover principal and generated by any uniformizer of $K$. The 
quotient $k = R/\m$ is a field which is called the \emph{residue 
field} of $K$.

\subsubsection*{Examples}

\noindent {\bf 1.}
Let $p$ be a prime number.
Recall that the \emph{$p$-adic valuation} of a nonzero integer $n$ is 
defined as the greatest integer $v$ such that $p^v$ divides $n$; it is 
often denoted by $v_p(n)$. This construction defines a function $v_p : 
\Z \backslash\{0\} \to \N$. We extend it to a function $\Q \to \Z \cup 
\{+\infty\}$ by setting:
$$v_p(0) = +\infty \quad \text{and} \quad 
\textstyle v_p(\frac a b) = v_p(a) - v_p(b)$$
for $a, b \in \Z$. One checks that $v_p$ satisfies the axioms of a
valuation, turning then $\Q$ into a discrete valuation field. 
A uniformizer of $(\Q, v_p)$ is $p$. Its rings of integers is the
ring $\Z_{(p)}$ consisting of fractions $\frac a b$ where $b$ is not
divisible by $p$. Its residue field is isomorphic to $\Z/p\Z$.

\medskip

\noindent {\bf 2.}
Let $k$ be any field and $K = k(t)$ be the field of univariate rational 
fractions over $k$. Given $f \in K$, $f\neq 0$, let $\ord(f)$ denote the 
order of vanishing of $f$ at $0$, \emph{i.e.} $\ord(f)$ is the unique 
integer for which one can write $f = t^{\ord(f)} \cdot g$ where $g \in 
k(t)$ is defined and does not vanish at $0$. This defines a function
$\ord : K^\star \to \Z$ that we extend to $K$ by letting $\ord(0) =
+ \infty$.
One then checks that $(K, \ord)$ is a discrete valuation field. Its
ring of integers consists of fractions $\frac f g$ where $f$ and $g$
are polynomials with $g(0) \neq 0$. A uniformizer of $(K, \ord)$ is
$t$ and its residue field is canonically isomorphic to $k$.

\subsection*{Completeness}

A discrete valuation field $(K, \val)$ is said \emph{complete} is it
complete\footnote{In the sense that all Cauchy sequences converge.} 
with respect to one (or equivalently all) $d_a$. 
Using the ultrametric triangle inequality~\eqref{eq:ultrametric}, we 
easily check that, assuming that $K$ is complete, a series $\sum_{n\geq 
0} u_n$ (with $u_n \in K$) converges if and only if the sequence 
$(u_n)_{n \geq 0}$ converges to $0$.

\medskip

Let $(K, \val)$ be a discrete valuation field and let $\hat K_a$ be the 
completion of the metric space $(K, d_a)$. One checks that $\hat K_a$ 
does not depend on $a$, so that we can denote it safely simply $\hat K$. 
Observe that the ring operations extend uniquely to $\hat K$, turning 
then it into a field. Similarly the continuous map $\val : K \to \Z 
\cup \{+\infty\}$ extends uniquely to $\hat K$, turning then $\hat K$
into a discrete valuation field. By construction $\hat K$ is moreover
complete. The ring of integers $\hat R$ of $\hat K$ can be seen as the
completion of $R$ or, alternatively, as the topological closure of $R$ 
in $\hat K$. Note moreover that a uniformizer of $K$ remains a 
uniformizer of $\hat K$ (since the valuation on $\hat K$ extends that
on $K$) and that the residue field of $\hat K$ is canonical isomorphic
to that of $K$.

\medskip

Elements in complete discrete valuation fields can be explicitely 
described as the values at a fixed uniformizer of particular power 
series.

\begin{prop}
\label{prop:decompCDVF}
Let $K$ be a complete discrete valuation field. Let $R$ be its ring 
of integers, $k$ be its residue field and $\pi$ be a fixed uniformizer. 
Let $S \subset R$ be a fixed complete system of representatives of $k$
and assume $0 \in S$. Then:
\begin{enumerate}[(1)]
\item 
any element $x \in R$ can be written uniquely as a converging sum:
\begin{equation}
\label{eq:decompCDVR}
x = s_0 + s_1 \pi + s_2 \pi^2 + \cdots
+ s_n \pi^n + \cdots
\end{equation}
with $s_n \in S$ for all $n \geq 0$
\item
any element $x \in K$ can be written uniquely as a converging sum:
$$x = s_v \pi^v + s_{v+1} \pi^{v+1} + s_{v+2} \pi^{v+2} + \cdots
+ s_n \pi^n$$
with $v \in \Z$, $s_n \in S$ for all $n \geq v$. We can moreover
require that $s_v \neq 0$, in which case we have $v = \val(x)$.
\end{enumerate}
\end{prop}

\begin{proof}
We only prove the first statement, the second being totally similar.
We first remark that the series \eqref{eq:decompCDVR} converges since
its general term $s_n \pi^n$ goes to $0$ when $n$ goes to infinity.

Assume first that we are given a decomposition~\eqref{eq:decompCDVR}.
Then $s_0$ has to be congruent to $x$ modulo $\pi$ and therefore is
uniquely determined since $S$ is by definition a complete set of
representatives of $k = R/\pi R$. Substrating $s_0$, dividing by
$\pi$ and applying the same reasoning, we find that $s_1$ is uniquely
determined as well. Repeating this argument again and again, we get
the unicity of the decomposition~\eqref{eq:decompCDVR}.

Now pick $x \in R$. Define $s_0$ as the unique element of $S$ which
is congruent to $x$ modulo $\pi$. Then $r_1 = \frac{x-s_0}{\pi}$ lies
in $R$. We can thus repeat the construction and define $s_1$ as the
unique element of $S$ which is congruent to $r_1$ modulo $\pi$.
We construct this way an infinite sequence $(s_n)_{n \geq 0}$ of
elements of $S$ with the property that 
$x \equiv s_0 + s_1 \pi + s_2 \pi^2 + \cdots + s_{n-1} \pi^{n-1}
\pmod{\pi^n}$ for all $n$. Passing to the limit (and noting that $\pi^n$
goes to $0$), we get~\eqref{eq:decompCDVR}.
\end{proof}

\subsubsection*{Examples}

\noindent {\bf 1.}
The field $\Q$ equipped with the $p$-adic valuation $v_p$ is \emph{not}
complete. Its completion is the field of $p$-adic numbers $\Q_p$. A
uniformizer of $\Q_p$ is $p$ and its residue field is $\Z/p\Z$. The
ring of integers of $\Q_p$ is usually denoted by $\Z_p$; its elements 
are the so-called $p$-adic integers.
According to Proposition~\ref{prop:decompCDVF}, any $p$-adic integer
can be uniquely written as a sum:
$$s_0 + s_1 p + s_2 p^2 + \cdots + s_n p^n + \cdots$$
with $s_n \in \{0, 1, \ldots, p{-}1\}$. It is the decomposition in
$p$-basis of a $p$-adic integer.

\medskip

\noindent {\bf 2.}
Similarly, the field $k(t)$ equipped with the valuation $\ord$ is
\emph{not} complete. Thanks to Proposition~\ref{prop:decompCDVF},
its completion consists of series of the shape:
$$s_v t^v + s_{v+1} t^{v+1} + s_{v+2} t^{v+2} + \cdots + s_n t^n + \cdots$$
with $v \in \Z$ and $s_n \in k$. It is therefore nothing but the field 
of univariate Laurent series over $k$, usually referred to as $k((t))$. 
Its rings of integers is the ring of power series over $k$, namely 
$k[[t]]$. Again its rings of integers is canonically isomorphic to $k$.

\subsection*{The Haar measure}

Let $(K, \val)$ be a \emph{complete} discrete valuation ring with ring 
of integers $R$ and residue field $k$. From now and until the end of
this appendix, \emph{we assume that $k$ is finite}.

The first part of Proposition \ref{prop:decompCDVF} shows that $R$ is 
homeomorphic to $k^{\N}$ (\emph{i.e.} the set of all sequences with 
coefficients in $k$) and therefore is compact.
Since $R$ carries in addition a group structure, it is endowed with a 
unique Haar measure $\mu$ normalized by $\mu(R) = 1$. This measure
extends uniquely to a Haar measure on $K$. Be careful nevertheless
that $\mu(K)$ is infinite.

Under the additional assumptions of this paragraph, it is quite 
convenient to normalize the norm $|\cdot|$ on $K$ by $|\pi| = \frac 1 
{\card k}$ where $\pi$ is any uniformizer. (If the valuation is
normalized so that it takes the value $1$, the above norm is the
norm $|\cdot|_{\card k}$ we have introduced before.) 
The above convention leads to the expected relation:
$$\mu(aE + b) = |a| \cdot \mu(E)$$
for all $a, b \in K$ and all measurable subset $E$ of $K$ (and where
$aE+b$ denotes of course the image of $E$ under the affine 
transformation $x \mapsto ax+b$).

\end{document}